\documentclass[a4paper,12pt,reqno]{amsart}
\usepackage{amssymb,amsmath,amsfonts,latexsym} 
\usepackage{hyperref}
\usepackage{ifthen}
\usepackage{graphicx}
\usepackage{float}
\usepackage{mathrsfs}
\usepackage{comment}
\usepackage{enumerate}
\usepackage[foot]{amsaddr}
 \usepackage[left=25mm,right=25mm,top=25mm,bottom=25mm,paper=a4paper]{geometry}

\theoremstyle{plain}

\newtheorem*{thm*}{Theorem}
\numberwithin{equation}{section}

\theoremstyle{definition}
\newcounter {own}
\def\theown {\thesection  .\arabic{own}}

{\qed\bigskip}

\newcounter{alphabet}

\newcommand{\R}{\mathbb{R}}

\newtheorem{theorem}{Theorem}[section]
\newtheorem{Cor}{Corollary}[section]
\newtheorem{definition}{Definition}[section]

\newtheorem{lemma}{Lemma}[section]
\newtheorem{prop}{Proposition}[section]
\newcounter{minutes}
\divide\time by 60
\newcounter{hours}
\multiply\time by 60 \addtocounter{minutes}{-\time}

\begin{document}
 

\title{ Hausdorff operators and fractional Hausdorff operators in the Dunkl setting}

\thanks{
File:~\jobname .tex,
          printed: \number\year-\number\month-\number\day,
          \thehours.\ifnum\theminutes<10{0}\fi\theminutes}


\author{Sumit Parashar}
 \address{Sumit Parashar, School of Mathematical and Statistical Sciences, Indian Institute of Technology Mandi, Kamand,
 Mandi, HP 175005, India \newline
Email: d22035@students.iitmandi.ac.in }

\author{ Saswata Adhikari $^\dagger$}
\address{Saswata Adhikari, School of Mathematical and Statistical Sciences, Indian Institute of Technology Mandi, Kamand,
 Mandi, HP 175005, India. \newline
 Email: saswata@iitmandi.ac.in  \newline
 $^\dagger$ {\tt Corresponding author}}



\begin{abstract}

 In this paper, we investigate the boundedness properties of the Dunkl-type Hausdorff operator on generalized Dunkl-type Morrey spaces and generalized Dunkl-type Campanato spaces. Furthermore, we introduce the fractional Dunkl-type Hausdorff operator and study its boundedness from $L^p(\mathbb{R}, d\mu_\alpha)$ to $L^q(\mathbb{R}, d\mu_\alpha)$. We also establish the boundedness of the fractional  Dunkl-type Hausdorff operator on generalized Dunkl-type Morrey spaces under suitable conditions. As a consequence of our main results, we derive several applications. Finally, we evaluate the exact norm of the Dunkl-type Hausdorff operator on Dunkl-type Morrey spaces and Dunkl-type Campanato spaces.
\end{abstract}

\maketitle
\pagestyle{myheadings}
\markboth{Sumit Parashar and Saswata Adhikari}{ Hausdorff operators and fractional Hausdorff operators in the Dunkl setting}
 \\

\subjclass \textbf{Mathematics Subject Classification[2010]} {Primary  42B10, 47G10; Secondary 47B38, 42B35}



\section{\bf{Introduction}} \label{sec:intro}

The study of integral operators and their boundedness properties on various function spaces has been one of the central themes in harmonic analysis and operator theory. Among these operators, the Hausdorff operator and its fractional variants are among the most important operators in harmonic analysis and play an important role due to their close connections with Hardy operators, adjoint Hardy operators, Hardy–Littlewood–Pólya operator, Cesàro operators, fractional Riemann-Liouville operators, and several other classical operators. In recent decades, the mapping properties of Hausdorff operators on Lebesgue, Morrey, Campanato, and other generalized function spaces have been studied extensively.

Initially, the (classical) Hausdorff operator for $\psi \in L^1_{loc}(0,\infty)$ is defined by 
\begin{align}
    \mathcal{H}_{\psi} f(x) := \int_0^\infty \frac{\psi(t)}{t} f\left(\frac{x}{t}\right) dt.\nonumber
\end{align}
Later, this operator was extended to the space $\mathbb{R}$ (see \cite{Andersen}). The Hausdorff operator on $\mathbb{R}$ is given by
\begin{align}
     \mathcal{H}_{\psi} f(x) := \int_{\R}\frac{\psi(t)}{|t|} f\left(\frac{x}{t}\right) dt.  \nonumber
\end{align}
Using the change of variables $y=\frac{x}{t}$, the above expression can equivalently be written as
\begin{align}
     \mathcal{H}_{\psi} f(x) := \int_{\R}\frac{\psi(x/y)}{|y|} f(y) dy.\label{eq:1.1}
\end{align}
Now for $\psi(t) = \frac{1}{|t|} \chi_{(1, \infty)}(|t|)$, the (classical) Hausdorff operator $\mathcal{H}_\psi$ reduces to the (classical) Hardy operator $\mathcal{H}$, which is defined as follows:
\begin{align}
    \mathcal{H} f(x) := \frac{1}{|x|} \int_{|t| < |x|} f(t) dt,\hspace {0.5 cm} x \in \R - \{0\}. \label{eq:1.01}
\end{align}
Similarly for $\psi(t) = \chi_{(0, 1]}(|t|)$, the  Hausdorff operator $\mathcal{H}_\psi$ reduces to the adjoint Hardy operator $\mathcal{H}^*$, which is defined as follows:
\begin{align}
    \mathcal{H}^* f(x) := \int_{|t| \geq |x|} \frac{f(t)}{|t|} dt. \label{eq:1.02}
\end{align}
Let $\psi(t)= \frac{1}{\text{max}\{1, |t|\}}$. Then the Hausdorff operator $\mathcal{H}_\psi$ reduces to the Hardy-Littlewood-P\'olya operator $\mathcal{T}$, which is defined as follows:
\begin{align}
    \mathcal{T} f(x) := \int_\R \frac{f(t)}{\text{max}\{|x|, |t|\}} dt, \hspace{0.5cm} x \in \R-\{0\}. \label{eq:1.03}
\end{align}
Furthermore, the Hausdorff operator can be extended to the fractional Hausdorff operator (see \cite{Lin2}), which is defined as
\begin{align}
    \mathcal{H}_{\psi, \beta} f(x) := \int_{\R}\frac{\psi(x/y)}{|y|^{1-\beta}} f(y) dy, \hspace{2cm} 0\leq \beta < 1. \label{eq: 1.2}
\end{align}
When $\beta =0$, then the fractional Hausdorff operator reduces to the Hausdorff operator defined in \eqref{eq:1.1}.

For $\psi(t)= \frac{1}{|t|^{1-\beta}} \chi_{(1, \infty)}(|t|)$, the fractional Hausdorff operator $\mathcal{H}_{\psi, \beta}$ reduces to the fractional Hardy operator $\mathcal{H}_\beta$, which is defined as follows:
\begin{align}
    \mathcal{H}_\beta f(x) := \frac{1}{|x|^{1-\beta}} \int_{|t| < |x|} f(t) dt, \hspace{0.5cm} x \in \R - \{0\}. \label{eq:1.401}
\end{align}
Similarly for $\psi(t)= \chi_{(0,1]}(|t|)$, the fractional Hausdorff operator $\mathcal{H}_{\psi,\beta}$ reduces to the adjoint fractional Hardy operator $\mathcal{H}_\beta^*$, which is defined as follows:
\begin{align}
    \mathcal{H}_\beta^* f(x) := \int_{|t| \geq |x|} \frac{f(t)}{|t|^{1-\beta}}  dt. \label{eq:1.402}
\end{align}
For $\psi(t)= \frac{1}{\text{max}\{1, |t|^{1-\beta}\}}$, $0 \leq \beta < 1$, the fractional Hausdorff operator $\mathcal{H}_{\psi, \beta}$ reduces to the fractional Hardy-Littlewood-P\'olya operator $\mathcal{T}_\beta$, which is defined as follows 
\begin{align}
    \mathcal{T}_\beta f(x) := \int_\R \frac{f(t)}{\text{max}\{ |x|^{1-\beta}, |t|^{1-\beta}\}} dt, \hspace{0.5cm} x \in \R - \{0\}. \label{eq: 1.07} 
\end{align}

Next, for $1 \leq p < \infty$ and a measurable function $\phi : \mathbb{R}^+ \to \mathbb{R}^+$, the generalized Morrey space $L^{p,\phi}(\mathbb{R})$ is defined to be the collection of all locally $p$-integrable functions $f$ on $\mathbb{R}$ such that
\begin{align*}
    L^{p, \phi}(\mathbb{R}) := \{ f \in L^p_{loc}(\mathbb{R}) : ||f||_{L^{p, \phi}} < \infty \},
\end{align*}
where
$$
||f||_{L^{p, \phi}} := \displaystyle \sup_{\substack{{r>0}\\{a \in \mathbb{R}}}} \frac{1}{\phi(r)} \left ( \frac{1}{r} \int_{B(a,r)} |f(x)|^p dx \right)^{\frac{1}{p}}. 
$$
In particular, when $\phi(r) = r^{- \frac{1}{q}}$, where $ 1 \leq p \leq q < \infty$, then the generalized Morrey space $L^{p, \phi}(\mathbb{R})$ coincides with the (classical)  Morrey space $L^{p, q}(\mathbb{R})$ and further when $p=q$, it reduces to the Lebesgue space $L^p(\mathbb{R})$.

Now, for $1 \leq p < \infty$ and a measurable function $\phi : \mathbb{R}^+ \to \mathbb{R}^+$, the generalized Campanato space, denoted by $\mathcal{L}^{p,\phi}(\mathbb{R})$, is defined as
\begin{align*}
    \mathcal{L}^{p,\phi}(\mathbb{R}) := \{ f \in L^p_{loc}(\mathbb{R}) : ||f||_{\mathcal{L}^{p,\phi}} < \infty \},
\end{align*}
where $$ ||f||_{\mathcal{L}^{p,\phi}} := \displaystyle \sup_{\substack{{r>0}\\{a \in \mathbb{R}}}} \frac{1}{\phi(r)} \left(\frac{1}{r} \int_{B(a, r)} | f(x) - f_{B(a,r)}|^p dx \right)^{1/p} $$  with $f_{B(a,r)} = \frac{1}{|B(a,r)|} \int_{B(a, r)} f(x) dx$, where $B(a,r)$ denotes the interval $(a-r, a+r)$. For $p =1$, $\mathcal{L}^{1, \phi}(\R)$ is the $BMO_\phi(\mathbb{R})$ space. Furthermore, when $\phi(r) = r^{- \frac{1}{q}}$, where $ 1 \leq p \leq q < \infty$, then the generalized Campanato space $\mathcal{L}^{p, \phi}(\mathbb{R})$ reduces to the (classical)  Campanato space $\mathcal{L}^{p, q}(\mathbb{R})$ and moreover, for $q =\infty$, $\mathcal{L}^{p, \infty}(\R)$ is the space of functions of bounded mean oscillation $BMO(\mathbb{R})$. 

The Hausdorff operator has been extensively studied in recent years, particularly its boundedness on the Lebesgue space as well as on the Hardy space (see \cite{el1, el2, el3}). In the case $0<p<1$, various $H^{p}$-estimates for the Hausdorff operator were obtained in \cite{Kanjin, el4,  Miyachi}. Further, we refer the reader to \cite{Chen, Chen2, Lin} for several recent developments and related results concerning Hausdorff operators. Burenkov and Liflyand  proved the boundedness of Hausdorff operators in the local and global Morrey-type spaces respectively in \cite{bl}. The boundedness of $H_{\psi}$ on BMO spaces was proved by Lerner and Liflyand \cite{el}. They obtained 
\begin{align}
    ||\mathcal{H}_\psi f
||_{BMO(\R)} \leq \left( \int_{\R} \frac{|\psi(y)|}{|y|} dy \right) ||f||_{BMO(\R)}, \nonumber
\end{align}
where $\psi$ satisfies the size condition 
\begin{align}
    \int_{\R}  \frac{|\psi(y)|}{|y|} dy < \infty. \nonumber
\end{align}
In \cite{Ruan}, J. Ruan et al. obtained the boundedness of $\mathcal{H_\psi}$ on Morrey spaces and Campanato spaces. Their results are presented below.
\begin{theorem} \cite{Ruan} \label{thm:1.1}
    Let $1 \leq p \leq q \leq \infty$. Then they had
    \begin{align}
        ||\mathcal{H}_\psi f ||_{L^{p,q}(\R)} \leq C_1 ||f||_{L^{p,q}(\R)}, \nonumber
    \end{align}
    where 
    \begin{align}
        C_1= \int_{\R} \frac{|\psi(y)|}{|y|^{\left( 1 - \frac{1}{q} \right)} }dy. \nonumber
    \end{align}
\end{theorem}
\begin{theorem} \cite{Ruan}
    Let $1 \leq p \leq q \leq \infty$. Then they had
    \begin{align}
        ||\mathcal{H}_\psi f ||_{\mathcal{L}^{p,q}(\R)} \leq C_1 ||f||_{\mathcal{L}^{p,q}(\R)}, \nonumber
    \end{align}
    where $C_1$ is the same as in Theorem \ref{thm:1.1}. 
\end{theorem}
Moreover, the boundedness of fractional Hausdorff operator from $L^p(\R)$ to $L^q(\R)$ was proved by G. Gao et al. in \cite{Gao}.

In $1989$, the Dunkl operators were introduced by C.F. Dunkl \cite{dunkl1989differential}, which are differential difference operators on the real line. For a real parameter $\alpha \geq -\frac{1}{2}$, the Dunkl operators are denoted by $\Lambda_\alpha$ and these are associated with the reflection group $Z_2$ on $\mathbb{R}$.  We also refer \cite{de1993dunkl, sy, rsl2} for more details on Dunkl theory. Using the Dunkl kernel, Dunkl defined the Dunkl transform $\mathcal{F}_\alpha$ in \cite{dunkl1992hankel}. R\"{o}sler in \cite{rosler1994bessel} showed that the Dunkl kernel verifies a product formula. This allows one to define the Dunkl translation $\tau_x^\alpha$, $x \in \mathbb{R}$ and as a result one has the Dunkl convolution.

For clarity, we shall now mention certain definitions and terms that are formally stated and elaborated upon in the later sections of this paper. For instance, we refer \eqref{eq: 3.1a} for generalized Dunkl-type Morrey spaces $L^{p,\phi}(\R, d\mu_\alpha)$, \eqref{eq: 7.1} for generalized Dunkl-type Campanato spaces $\mathcal{L}^{p, \phi}(\R, d\mu_\alpha)$, as well as \eqref{eq:3.1a} for the Dunkl-type Hausdorff operator $\mathcal{H}_\psi^\alpha$, \eqref{eq:4.1a} for   fractional Dunkl-type Hausdorff operators $\mathcal{H}_{\psi, \beta}^\alpha$. 

The boundedness of the Dunkl-type Hausdorff operator on the Dunkl-type 
Lebesgue space $L^1(\mathbb{R},d\mu_\alpha)$ was proved in 
\cite{dah}. Later, the boundedness of the Dunkl-type Hausdorff operator was extended to $L^p(\mathbb{R},d\mu_\alpha)$, 
where $1\leq p<\infty$ in \cite{dah2}. More precisely, they 
obtained the following results.
 \begin{theorem}
Let $\psi \in L^1(\mathbb{R})$. Then the Dunkl-type Hausdorff operator 
$\mathcal{H}_\psi^\alpha$ is bounded from $L^1(\mathbb{R}, d\mu_\alpha)$ into itself. Indeed, one has
\begin{align}
\|\mathcal{H}_\psi^\alpha f\|_{L^1(\mathbb{R},d\mu_\alpha)}
\leq \|\psi\|_{L^1}\|f\|_{L^1(\mathbb{R},d\mu_\alpha)},~\forall f \in L^1(\mathbb{R}, d\mu_\alpha), \nonumber
\end{align}
\end{theorem}
   
  \begin{theorem}
Let $\psi$ be a measurable function on $\mathbb{R}$ such that
\begin{align}
K_{p,\alpha,\psi}
= \int_{\mathbb{R}} |\psi(t)|\, |t|^{(2\alpha+2)\left(\frac{1}{p}-1\right)}\,dt < \infty, \nonumber
\end{align}
for some $\alpha > -\frac{1}{2}$ and $p \in [1,\infty)$. Then the Dunkl-type Hausdorff operator
$\mathcal{H}_\psi^\alpha$ is bounded on $L^p(\mathbb{R},d\mu_\alpha)$. Moreover,
\begin{align}
\|\mathcal{H}_\psi^\alpha f\|_{L^p(\mathbb{R},d\mu_\alpha)}
\leq K_{p,\alpha,\psi}
\|f\|_{L^p(\mathbb{R},d\mu_\alpha)}. \nonumber
\end{align}
\end{theorem}

Motivated by these results, in the present paper we study the boundedness of the Dunkl-type Hausdorff operator $\mathcal{H}_\psi^\alpha$ on generalized Dunkl-type Morrey spaces $L^{p, \phi}(\R, d\mu_\alpha)$ and generalized Dunkl-type Campanato spaces $\mathcal{L}^{p, \phi}(\R, d\mu_\alpha)$. When $\alpha= -\frac{1}{2}$, these results reduce to the boundedness of the Hausdorff operator on the generalized Morrey and generalized Campanato spaces in the classical setting, which are new results in their own right. We also introduce, for the first time, the fractional Dunkl-type Hausdorff operator $\mathcal{H}_{\psi, \beta}^\alpha$. Furthermore, we establish the boundedness of the  fractional Dunkl-type Hausdorff operator from $L^p(\R, d\mu_\alpha)$ to $L^q(\R, d\mu_\alpha)$ and from $L^{p_1,\phi}(\R, d\mu_\alpha)$ to $L^{p_2,\omega}(\R, d\mu_\alpha)$, under suitable conditions on the parameters. To the best of our knowledge, the boundedness of the fractional Hausdorff operator has not yet been studied even on the classical Morrey space. In particular, our results give a new contribution to the classical case as well.

As applications of our main results, we also obtain the boundedness of several operators as special cases of the Dunkl-type Hausdorff operator and its fractional version. Specifically, we derive the boundedness of the Hardy operator, the adjoint Hardy operator, and the Hardy–Littlewood–P\'olya operator on Morrey spaces and Campanato spaces, as well as the boundedness of their fractional versions on Lebesgue spaces and Morrey spaces associated with the Dunkl operator on the real line. This illustrates the unifying role of the Hausdorff operator framework in recovering and extending known boundedness results within the Dunkl setting. 

Now the paper is organized as follows. Section \ref{sec:prelim} provides a brief review of Dunkl theory on the real line along with some known results. In Section \ref{sec:3}, we show the boundedness of Dunkl-type Hausdorff operator on generalized Dunkl-type Morrey spaces and generalized Dunkl-type Campanato spaces. In Section \ref{sec:4}, we define the  fractional Dunkl-type Hausdorff operator and show its boundedness from $L^p(\R, d\mu_\alpha)$ to $L^q(\R, d\mu_\alpha)$ and from $L^{p_1,\phi}(\R, d\mu_\alpha)$ to $L^{p_2,\omega}(\R, d\mu_\alpha)$. In Section \ref{Sec: 5}, as applications of our main results, we also obtain the boundedness of several operators. Finally, in Section \ref{sec:6}, we find the exact norm of Dunkl-type Hausdorff operator on Dunkl-type Morrey spaces and Dunkl-type Campanato spaces.

\section{\bf{Preliminaries}} \label{sec:prelim}
 The Dunkl operator associated with  a fixed real number $\alpha \geq -\frac{1}{2}$ is defined by
 \begin{align*}
     \Lambda_\alpha f(x) = \frac{df}{dx}(x) + \frac{2 \alpha +1}{x} \frac{f(x) - f(-x)}{2},  ~x \in \mathbb{R},
 \end{align*}
 where $f \in C^\infty(\mathbb{R})$. When $\alpha = -\frac{1}{2}$ the Dunkl derivative $\Lambda_{-\frac{1}{2}}$ is equal to the classical derivative $\frac{d}{dx}$.\\
 For $\alpha \geq - 1/2$ and $\lambda \in \mathbb{C}$, the initial value problem $\Lambda_\alpha f(x) = \lambda f(x)$; $f(0) =1$ has a unique solution $E_\alpha(\lambda x)$ called Dunkl kernel (see \cite{dunkl1991integral, de1993dunkl, sifi2002generalized}), which is  given by 
 \begin{align*}
     E_\alpha(\lambda x) = \mathcal{J}_\alpha (\lambda x) + \frac{\lambda x}{2(\alpha + 1)} \mathcal{J}_{\alpha+1} (\lambda x), ~ x \in \mathbb{R}, 
 \end{align*}
 where
 \begin{align*}
     \mathcal{J}_\alpha (\lambda x) := \Gamma (\alpha +1) \sum_{n=0}^\infty \frac{(\lambda x / 2)^{2n}}{n! \Gamma(n + \alpha + 1)}
 \end{align*}
 is the modified spherical Bessel function of order $\alpha$. Note that $E_{- 1/ 2}(\lambda x) = e^{\lambda x}$. \\
 The weighted Lebesgue measure $\mu_\alpha$ on $\mathbb{R}$ is given by 
 $$
 d\mu_\alpha(x) :=  A_\alpha |x|^{2 \alpha + 1} dx, \hspace{1cm}  where ~~  A_\alpha = (2 ^ {\alpha + 1 } \Gamma(\alpha + 1))^{-1}.
 $$
 For $1 \leq p \leq \infty$, the space $L^p(\mathbb{R}, d\mu_\alpha)$ denotes the class of  complex-valued measurable functions $f$ on $\mathbb{R}$ such that $$ ||f||_{L^p(\mathbb{R}, d\mu_\alpha)} :=\left(\int_{\mathbb{R}} |f(x)|^p d\mu_\alpha(x)\right)^{\frac{1}{p}} < \infty ,~ ~ if ~ ~ p \in [1, \infty)$$
 and $$ ||f||_{L^{\infty}(\mathbb{R}, d\mu_\alpha)} := ess \displaystyle \sup_{x \in \mathbb{R}} |f(x)| < \infty, ~ ~ if ~ ~ p = \infty.$$\\
Using the Dunkl kernel, one can define the Dunkl transform on $\mathbb{R}$ (see \cite{de1993dunkl}) as follows:\\
 The Dunkl transform of a function $f \in L^1(\mathbb{R}, d\mu_\alpha)$, is given by 
 \begin{align*}
     \mathcal{F}_\alpha (f)(\lambda) := \int_{\mathbb{R}} E_\alpha(- i \lambda x) f(x) d\mu_\alpha(x), ~ \lambda \in \mathbb{R}.
 \end{align*}
 The above integral makes sense as $|E_\lambda (i x)| \leq 1$, $\forall x \in \mathbb{R}$ (see \cite{rosler1994bessel}). Note that $\mathcal{F}_{- 1/ 2}$ agrees with definition of the classical Fourier transform $\mathcal{F}$, which is given by
 \begin{align*}
     \mathcal{F}(f)(\lambda) := (2 \pi)^{-1/2} \int_{\mathbb{R}} e^{-i \lambda x} f(x) dx, ~ ~ \lambda \in \mathbb{R}.
 \end{align*}
In the below, we list some of the properties of the Dunkl transform $\mathcal{F}_\alpha$. \begin{theorem}
    
\cite{trime2002paley}
 \begin{itemize}
     \item[(i)] For all $ f \in L^1(\mathbb{R}, d\mu_\alpha)$, one has $|| \mathcal{F}_\alpha (f)||_{L^\infty(\mathbb{R}, d\mu_\alpha)} \leq ||f||_{L^1(\mathbb{R}, d\mu_\alpha)}$.
     \item [(ii)] For all $f \in \mathcal{S}(\R)$,
     \begin{align*}
         \mathcal{F}_\alpha ( \Lambda_\alpha f)(\lambda) = i \lambda \mathcal{F}_\alpha (f)(\lambda), ~~ \lambda \in \mathbb{R}.
     \end{align*}
     \item [(iii)] The Dunkl transform $\mathcal{F}_\alpha$ is an isometric isomorphism from $L^2(\mathbb{R}, d\mu_\alpha)$ onto $L^2(\mathbb{R}, d\mu_\alpha)$. In particular, it satisfies the Plancherel formula i.e.,
     \begin{align*}
         ||\mathcal{F}_\alpha (f)||_{L^2(\mathbb{R}, d\mu_\alpha)} = ||f||_{L^2(\mathbb{R}, d\mu_\alpha)}.
     \end{align*}
     \item[(iv)] If $f \in L^1(\mathbb{R}, d\mu_\alpha)$ with $\mathcal{F}_\alpha(f) \in L^1(\mathbb{R}, d\mu_\alpha)$, then one has the inversion formula 
     \begin{align*}
         f(x) = \int_{\mathbb{R}} \mathcal{F}_\alpha (f) (\lambda) E_\alpha(i \lambda x) d\mu_\alpha(\lambda), ~~ a.e.~~ x \in \mathbb{R}.
     \end{align*}
 \end{itemize}
 \end{theorem} 
Now we will denote $B(x,R) = \{y \in \mathbb{R} : |y| \in \,]max\{0, |x|-R\}, |x| + R[\,\}$, $R > 0 $, if $x\neq 0$ and $B(0, R) = \,]-R, R[\, $. Then $\mu_\alpha (B(0, R)) = b_\alpha R^{d_\alpha}$, where $b_\alpha = [2^{\alpha + 1}(\alpha + 1) \Gamma(\alpha + 1)]^{-1}$ and $d_\alpha = 2 \alpha + 2$. We observe that when $|x| \leq r$ then $B(x,r) = B(0, |x| + r) \subseteq B(0, 2r)$ and when $|x| > r$ then $B(x,r)= \{ y \in \mathbb{R} : |y| \in \,]|x|-r, |x|+r[\, \}\subseteq B(0, |x| +r) \subseteq B(0, 2|x|)$.\\
Next for $x, y, z \in \mathbb{R}$, we put
 \begin{align}
 W_\alpha(x, y, z) = (1 - \sigma_{x, y, z} + \sigma_{z, x, y} + \sigma_{z, y, x} ) \kappa_\alpha(|x|, |y|, |z|)  \label{bessel}
 \end{align}
 where 
\begin{align}
 \sigma_{x, y, z} &:=
 \begin{cases}
     \frac{x^2 + y^2 - z^2}{2xy}, & \text{if} ~ x, y  \in \mathbb{R}\setminus \{0\},\\
     0, & \text{otherwise}, 
 \end{cases} \label{eq: 2.24}
\end{align} 
    and $\kappa_\alpha$ is the Bessel kernel given by 
\begin{align}    
    \kappa_\alpha(|x|, |y|, |z|) :=
    \begin{cases}
        D_\alpha \frac{([(|x| + |y|)^2 - z^2][z^2 - (|x| - |y|)^2])^{\alpha - \frac{1}{2}}}{|xyz|^{2\alpha}}, & \text{if} ~ |z| \in S_{x, y},\\
        0, & \text{otherwise},
    \end{cases} \nonumber
\end{align}
Where $D_\alpha = (\Gamma(\alpha + 1))^2/ (2^{\alpha - 1} \sqrt{\pi} \Gamma(\alpha + \frac{1}{2}))$ and $ S_{x, y} = [||x|-|y||,|x|+|y|]$.
Furthermore, if $|z| \in S_{x,y}$ then $|\sigma_{x,y,z}| \leq 1$,$|\sigma_{z,x,y}| \leq 1$ and $|\sigma_{z,y,x}| \leq 1$ ( see \cite{rosler1994bessel}).

Next using $W_\alpha(x,y,z)$, the signed measure $\nu_{x, y}$ on $\mathbb{R}$ is given by
  \begin{align}
      d\nu_{x, y}(z) &= \begin{cases}
      W_\alpha(x, y, z) d\mu_\alpha(z), & \text{if} ~ x, y \in \R \backslash \{0\},\\ d\delta_x(z), & \text{if} ~ y = 0,\\ d\delta_y(z), & \text{if}~ x = 0,
  \end{cases} \nonumber
\end{align} 
where the signed kernel  $W_\alpha$ is an even function with respect to all variables satisfying the following properties ( see \cite{rosler1994bessel}):$$ W_\alpha(x, y, z) = W_\alpha(y, x, z) = W_\alpha(-x, z, y),$$
  $$ W_\alpha(x, y, z) = W_\alpha(-z, y, -x) = W_\alpha(-x, -y, -z)$$ and $$ \int_{\mathbb{R}} |W_\alpha(x, y, z)| d\mu_\alpha(z) \leq 4.$$
     Let $x, y \in \R$ and $f$ be a continuous function on $\R$. Then the Dunkl translation operator of $f$ is defined by
     \begin{align}
         \tau_x^\alpha f(y) = \int_{\R} f(z) d\nu_{x,y}(z).  \label{2.1a}
     \end{align}
     The Dunkl translation operator satisfies the following properties:
      \begin{prop} \cite{mourou2001transmutation} \label{pro: 2.1}
        \begin{itemize}
            \item [(i)] $\tau_x^\alpha, ~x \in \mathbb{R}$ is a bounded linear operator on $C^\infty(\mathbb{R})$
            \item [(ii)] For all $f \in C^\infty(\mathbb{R})$ and $x,y \in \mathbb{R}$, one has 
            \begin{align*}
                \tau_x^\alpha f(y) = \tau_y^\alpha f(x), \tau_0^\alpha f(x) = f(x), ~ \tau_x^\alpha \circ \tau_y^\alpha = \tau_y^\alpha \circ \tau_x^\alpha.
            \end{align*} 
        \end{itemize}
      \end{prop}

 If $f, g \in L^2(\mathbb{R}, d\mu_\alpha)$, then 
            \begin{align}
                \int_{\mathbb{R}} \tau_x^\alpha f(y) g(y) d\mu_\alpha(y) = \int_{\mathbb{R}} f(y) \tau_{-x}^\alpha g(y) d\mu_\alpha(y), ~ \forall x \in \mathbb{R}. \nonumber
            \end{align}
     
  
Suppose $f$, $g$ are continuous functions on $\mathbb{R}$ with compact support. The generalized convolution $f*_\alpha g$ is defined by 
\begin{align}
    f*_\alpha g(x) := \int_{\mathbb{R}} f(y) \tau_y^\alpha g(x) d\mu_\alpha(y). \label{eq:dc}
\end{align}
The generalized convolution $*_\alpha$ is associative and commutative \cite{rosler1994bessel}. We also mention in the following some useful properties of the generalized convolution $*_\alpha$.
\begin{theorem}\cite{soltani2004lp} \label{tm: 2.1}
\begin{enumerate}
    \item[(i)]   For all $x \in \mathbb{R}$ and $p \geq 1$, the generalized translation operator $\tau_x^\alpha$ is bounded from $L^p(\mathbb{R}, d\mu_\alpha)$ to $L^p(\mathbb{R}, d\mu_\alpha)$ i.e.
      \begin{align*}
          ||\tau_x^\alpha f||_{L^p(\mathbb{R}, d\mu_\alpha)} \leq 4 ||f||_{L^p(\mathbb{R}, d\mu_\alpha)}.
       \end{align*}
    \item [(ii)] If $f\in L^{1}(\mathbb{R}, d\mu_\nu)$, then 
\begin{eqnarray*}
\mathcal{F}_{\alpha} (\tau_x^\alpha f)(\lambda)=E_{\alpha}(i\lambda x)\mathcal{F}_{\alpha}(f)(\lambda),~x,\lambda\in\mathbb{R}.
\end{eqnarray*}
\item [(iii)] If  $f\in L^{1}(\mathbb{R}, d\mu_\alpha)$ and $g\in L^{2}(\mathbb{R}, d\mu_\alpha)$, then 
\begin{eqnarray*}
\mathcal{F}_{\alpha}(f\ast_\alpha g)=\mathcal {F}_{\alpha} (f)\mathcal{F}_{\alpha} (g).
\end{eqnarray*}
 \item[(iv)] Assume that $p,q,r \in [1, \infty]$ satisfying $\frac{1}{p} + \frac{1}{q} = 1 + \frac{1}{r}$  (the Young condition). Then the map $(f,g) \to f *_\alpha g$ defined on $C_c(\mathbb{R}) \times C_c(\mathbb{R})$, extends to a continuous map from
 $L^p(\mathbb{R}, d\mu_\alpha) \times L^q(\mathbb{R}, d\mu_\alpha)$ to $L^r(\mathbb{R}, d\mu_\alpha)$ and one has
\begin{align*}
    ||f *_\alpha g ||_{L^r(\mathbb{R}, d\mu_\alpha)} \leq 4 ||f||_{L^p(\mathbb{R}, d\mu_\alpha)} ||g||_{L^q(\mathbb{R}, d\mu_\alpha)}.
\end{align*} 
\end{enumerate}
\end{theorem}

Now we will recall a few definitions and the known facts in the following, which will be useful in our paper.

A function $\theta : \,]0,+\infty[\, \to \,]0,+\infty[\,$ is said to be almost increasing (almost decreasing) if there exists a constant $C>0$ such that $\theta(r) \leq C \theta(s) ~ (\theta(r) \geq C\theta(s))$  for $r \leq s$.
\begin{definition} \cite{spsa}
 Let $\phi : \R^{+} \to \R^{+}$ be a function which is measurable. For $1\leq p < \infty$, the generalized Dunkl-type Morrey space $L^{p,\phi}(\R, d\mu_\alpha)$ is defined to be the class of all locally $p$-integrable functions $f$ on $\mathbb{R}$ such that 
    \begin{align}
    ||f||_{L^{p, \phi}(\R, d\mu_\alpha)} :=  \sup_{\substack{{r > 0}\\{x \in \R}}} \frac{1}{\phi(r)}\left(\frac{1}{r^{d_\alpha}}\int_{B(0,r)} \tau_x^\alpha|f|^p(y) d\mu_\alpha(y) \right)^{\frac{1}{p}} < \infty. \label{eq: 3.1a}
\end{align}
In particular, when $\phi(r) = r^{- \frac{d_\alpha}{q}}$, where $ 1 \leq p \leq q < \infty$, then the generalized Dunkl-type Morrey space $L^{p, \phi}(\mathbb{R}, d\mu_\alpha)$ coincides with the Dunkl-type  Morrey space $L^{p, q}(\mathbb{R}, d\mu_\alpha)$ (see \cite{sps}) and moreover, when $p=q$, it reduces to the Dunkl-type Lebesgue space $L^p(\mathbb{R}, d\mu_\alpha)$.

\end{definition}
\begin{definition} \cite{spsa2}
    For $1 \leq p < \infty$ and a function $\phi : (0, +\infty) \to (0, + \infty)$, we define the generalized Dunkl-type Campanato space $\mathcal{L}^{p,\phi}(\mathbb{R}, d\mu_\alpha)$ to be the space of all locally p-integrable functions $f$ on $\mathbb{R}$ such that 
\begin{align}
    &||f||_{\mathcal{L}^{p,\phi}(\mathbb{R}, d\mu_\alpha)} \nonumber
    \\
    &:= \sup_{\substack{{r>0}\\{x \in \R}}} \frac{1}{\phi(r)} \left(\frac{1}{\mu_{\alpha}(B(0,r))} \int_{B(0,r)} |\tau_x^\alpha f(y) -f^\alpha_{B(0,r)}(x) |^p d\mu_\alpha(y) \right)^{\frac{1}{p}} < \infty, \label{eq: 7.1}
\end{align} 
where 
\begin{align}
f^\alpha_{B(0,r)}(x) = \frac{1}{\mu_\alpha(B(0,r))} \int_{B(0,r)} \tau_x^\alpha f(y) d\mu_\alpha(y). \label{eq:2.6ab}
\end{align}
If $p=1$, then $\mathcal{L}^{1,\phi}(\R,d\mu_\alpha)= BMO_{\phi}(\R, d\mu_\alpha)$( see \cite{spsa}). Furthermore, when $\phi(r) = r^{- \frac{d_\alpha}{q}}$, where $ 1 \leq p \leq q < \infty$, then the generalized Dunkl-type Campanato space $\mathcal{L}^{p, \phi}(\mathbb{R}, d\mu_\alpha)$ reduces to the Dunkl-type  Campanato space $\mathcal{L}^{p, q}(\mathbb{R}, d\mu_\alpha)$ and for $ q =\infty$, $\mathcal{L}^{p, \infty}(\R, d\mu_\alpha)$ is the Dunkl-type BMO space $BMO(\mathbb{R}, d\mu_\alpha)$ (\cite{guliyev2009fractional}).
\end{definition}
In order to define $L^{p, \phi}(\R, d\mu_\alpha)$ and  $\mathcal{L}^{p,\phi}(\R, d\mu_\alpha)$, we assume  that $\phi$ is almost decreasing and $r^{\frac{d_\alpha}{p}}\phi(r)$ is almost increasing, so that there exists a constant $C_1 > 0$ such that 
\begin{align}
    \frac{1}{2} \leq \frac{r}{s} \leq 2 \implies \frac{1}{C_1} \leq \frac{\phi(r)}{\phi(s)} \leq C_1, \label{eq: 3.3a}
\end{align} 
i.e. $\phi$ satisfies the doubling condition.

\begin{lemma} \cite{sps}
    If  $0 < \beta < d_\alpha$, then $|.|^{\beta-d_\alpha} \in L^{s,t}(\R, d\mu_\alpha)$, where $1 \leq s < t= \frac{d_\alpha}{d_\alpha -\beta}$.
\end{lemma}

     \begin{lemma} \cite{lk}\label{7.1}
      For $f \in L^1(\mathbb{R}, d\mu_\alpha)$ and $g \in L^p(\mathbb{R}, d\mu_\alpha)$, $1 \leq p < \infty$,
\begin{align}
\tau_{x_0}^\alpha(f *_\alpha g) = \tau_{x_0}^\alpha f *_\alpha g = f *_\alpha \tau_{x_0}^\alpha g,  ~~x_0 \in \mathbb{R}. \nonumber
\end{align}
    \end{lemma}
    \begin{lemma} (\cite{abdelkefi2007dunkl}, \cite{guliyev2010p}) \label{lem: 2.2}
        Support of $\tau_x^\alpha \chi_{B(0,r)}$ is contained in the ball $B(x,r)$  for all $x \in \R$. Moreover, 
        \begin{align}
            0 \leq \tau_x^\alpha \chi_{B(0, r)}(y) \leq min \left\{ 1, \frac{2 c_\alpha}{2\alpha + 1} \left( \frac{r}{|x|} \right)^{2 \alpha + 1} \right \} , ~\forall~ y \in B(x, r), \nonumber
        \end{align}
        where $c_\alpha = \frac{ \Gamma(\alpha + 1)}{\sqrt{\pi} \Gamma(\alpha + \frac{1}{2})}$. In fact, when $|x| \leq r$, one has $0 \leq \tau_x^\alpha \chi_{B(0,r)}(y) \leq 1$ and when $|x| > r$, one has $0 \leq \tau_x^\alpha \chi_{B(0,r)}(y) \leq \frac{2 c_\alpha}{2 \alpha +1} (\frac{r}{|x|})^{2 \alpha + 1}$.
    \end{lemma}

\section{\bf{Hausdorff operators on Morrey spaces and Campanato spaces associated with Dunkl operator on the real line}} \label{sec:3}
For a function $\psi \in L^1(\R)$, the Dunkl-type Hausdorff operator can be defined (see \cite{dah}) as follows:
\begin{align}
    \mathcal{H}_{\psi}^\alpha f(x) := \int_\R \frac{\psi(t)}{|t|^{2\alpha +2}} f \left(\frac{x}{t} \right)  dt, ~f \in L^1(\R, d\mu_\alpha), ~x \in \R. \label{eq:3.1a}
\end{align}
For $\alpha= {-1/2}$, $\mathcal{H}_{\psi}^\alpha$ reduces to the classical Hausdorff operator (see \eqref{eq:1.1}).

The goal of this section is to prove the boundedness of the Dunkl-type Hausdorff operator on the generalized Dunkl-type Morrey space and the generalized Dunkl-type Campanato space. Before proving these results , first we prove the following lemma that relates the Dunkl-type Hausdorff operator and the generalized translation operator.
\begin{lemma} \label{lem: 3.1a1}
    For all $x,y \in \R$ and $f \in L^1(\R, d\mu_\alpha)$, we have
    \begin{align}
    \tau_x^\alpha \mathcal{H}_\psi^\alpha f(y) = \int_\R \frac{\psi(t)}{|t|^{2 \alpha + 2}} \tau_{x/t}^\alpha f\left( y/t \right) dt. \nonumber
     \end{align}
\end{lemma}
\textbf{Proof.} For $t \neq 0$, first we observe from \eqref{eq: 2.24},
\begin{align}
 \sigma_{x, y, tz} &:=
 \begin{cases}
     \frac{x^2 + y^2 - (tz)^2}{2xy}, & \text{if} ~ x, y  \in \mathbb{R}\setminus \{0\},\\
     0, & \text{otherwise}
 \end{cases}\nonumber
 \\
 &= \begin{cases}
     \frac{(x/t)^2 + (y/t)^2 - z^2}{2xy/{t^2}}, & \text{if} ~ x, y  \in \mathbb{R}\setminus \{0\},\\
     0, & \text{otherwise}
 \end{cases} \nonumber
 \\
 & = \sigma_{\frac{x}{t}, \frac{y}{t}, z}, \nonumber
\end{align} 
and similarly
\begin{align}    
    \kappa_\alpha(|x|, |y|, |tz|) &:=
    \begin{cases}
        D_\alpha \frac{([(|x| + |y|)^2 - (tz)^2][(tz)^2 - (|x| - |y|)^2])^{\alpha - \frac{1}{2}}}{|xytz|^{2\alpha}}, & \text{if} ~ |tz| \in S_{x, y},\\
        0, &\text{otherwise}
    \end{cases} \nonumber 
    \\
    & = \begin{cases}
        D_\alpha \frac{([(|x/t| + |y/t|)^2 - z^2][z^2 - (|x/t| - |y/t|)^2])^{\alpha - \frac{1}{2}}}{t^{2\alpha + 2} \left|\frac{xyz}{t^2}\right|^{2\alpha}}, & \text{if} ~ |z| \in S_{\frac{x}{t}, \frac{y}{t}},\\
        0, & \text{otherwise}
    \end{cases} \nonumber 
    \\
    & = \frac{\kappa_\alpha(|x/t|, |y/t|, |z|)}{|t|^{2\alpha + 2}}, \nonumber
\end{align}
where $S_{\frac{x}{t}, \frac{y}{t}} = \left[ \frac{||x|-|y||}{t},\frac{|x|+|y|}{t} \right]$. Then we obtain from \eqref{bessel},
\begin{align}
    W_\alpha(x, y, tz) &= (1 - \sigma_{x, y, tz} + \sigma_{tz, x, y} + \sigma_{tz, y, x} ) \kappa_\alpha(|x|, |y|, |tz|) \nonumber
    \\
    &= \frac{(1 - \sigma_{x/t, y/t, z} + \sigma_{z, x/t, y/t} + \sigma_{z, y/t, x/t} ) \kappa_\alpha(|x/t|, |y/t|, |z|)}{|t|^{2\alpha + 2}} \nonumber
    \\
    &= \frac{W_\alpha(x/t, y/t, z)}{|t|^{2 \alpha + 2}}. \label{eq: 3.1ba}
\end{align}
Consequently, by applying the formula for Dunkl translation (see \eqref{2.1a}) and the definition of the Dunkl-type Hausdorff operator (see \eqref{eq:3.1a}) together with Fubini's theorem, we deduce
\begin{align}
    \tau_x^\alpha \mathcal{H}_\psi^\alpha f (y) &= \int_\R \mathcal{H}_\psi^\alpha f(z) W_\alpha(x,y,z) d\mu_\alpha(z) \nonumber
    \\
     &= \int_\R \left( \int_\R \frac{\psi(t)}{|t|^{2 \alpha + 2}} f\left( \frac{z}{t} \right) dt \right)  W_\alpha(x,y,z) d\mu_\alpha(z) \nonumber
     \\
      &= \int_\R \frac{\psi(t)}{|t|^{2 \alpha + 2}} \left( \int_\R  f\left( \frac{z}{t} \right) W_\alpha(x,y,z) d\mu_\alpha(z)\right) dt.    \nonumber
     \end{align}
    By making the change of variables $\frac{z}{t} = u$ and using \eqref{eq: 3.1ba}, we obtain
    \begin{align}
       \tau_x^\alpha \mathcal{H}_\psi^\alpha f (y) &= \int_\R \frac{\psi(t)}{|t|^{2 \alpha + 2}} \left( \int_\R  f(u) W_\alpha(x,y,tu) |t|^{2 \alpha +2} d\mu_\alpha(u)\right) dt    \nonumber
    \\
    &= \int_\R \frac{\psi(t)}{|t|^{2 \alpha + 2}} \left( \int_\R  f(u) W_\alpha(x/t,y/t,u) d\mu_\alpha(u)\right) dt    \nonumber
    \\
      &= \int_\R \frac{\psi(t)}{|t|^{2 \alpha + 2}} \tau_{x/t}^\alpha f\left( y/t \right) dt.    \nonumber
\end{align}
 This completes the proof. \qed
 \begin{theorem} \label{thm:3.1a}
     Let $1 \leq p < \infty$ and $\phi$ be a positive measurable function on $\R^{+}$. Then the Dunkl-type Hausdorff operator is bounded on the generalized Dunkl-type Morrey space i.e.
     \begin{align}
    ||\mathcal{H}_\psi^\alpha f||_{L^{p,\phi}(\R, d\mu_\alpha)} \leq K_{\phi, \psi} ||f||_{L^{p,\phi}(\R, d\mu_\alpha)}, \nonumber
\end{align}
    where 
    \begin{align}
       K_{\phi,\psi} :=  \sup_{r > 0} \frac{1}{\phi(r)}\int_\R \frac{|\psi(t)|}{|t|^{d_\alpha}} \phi\left(\frac{r}{|t|} \right) dt < \infty. \nonumber
    \end{align}
 \end{theorem}
 \textbf{Proof.} From the definition of the generalized Dunkl-type Morrey norm, we write
\begin{align}
    ||\mathcal{H}_{\psi}^\alpha f ||_{L^{p,\phi}(\R, d\mu_\alpha)} &= \sup_{\substack{{r>0}\\{x \in \R}}} \frac{1}{\phi(r)}\frac{1}{r^{\frac{d_\alpha}{p}}} \left(\int_{B(0,r)} \tau_x^\alpha |\mathcal{H}_\psi^\alpha f|^p(y) d\mu_\alpha(y)\right)^{\frac{1}{p}}. \label{eq:3.1a1}
\end{align}
Now, we evaluate the integral term in \eqref{eq:3.1a1} using Minkowski inequality
\begin{align}
    &\left(\int_{B(0,r)} \tau_x^\alpha |\mathcal{H}_\psi^\alpha f|^p(y) d\mu_\alpha(y)\right)^{\frac{1}{p}} \nonumber
    \\
    &= \left(\int_{\R}  |\mathcal{H}_\psi^\alpha f|^p(y) \tau_{-x}^\alpha \chi_{B(0,r)}(y) d\mu_\alpha(y)\right)^{\frac{1}{p}} \nonumber
    \\
    & = \left(\int_{\R}  \left| \int_\R \frac{\psi(t)}{|t|^{2\alpha + 2}} f \left(\frac{y}{t} \right) dt\right|^p \tau_{-x}^\alpha \chi_{B(0,r)}(y) d\mu_\alpha(y)\right)^{\frac{1}{p}} \nonumber
    \\
    & \leq \left(\int_{\R} \left( \int_\R \frac{|\psi(t)|}{|t|^{2\alpha + 2}} \left|f \left(\frac{y}{t} \right) \right| \left(\tau_{-x}^\alpha \chi_{B(0,r)}(y) \right)^{1/p} dt \right)^p  d\mu_\alpha(y)\right)^{\frac{1}{p}} \nonumber
    \\
    &  \leq \int_{\R} \frac{|\psi(t)|}{|t|^{2\alpha + 2}} \left( \int_\R  \left|f \left(\frac{y}{t} \right) \right|^p \tau_{-x}^\alpha \chi_{B(0,r)}(y)  |y|^{2 \alpha + 1} dy \right)^{\frac{1}{p}} dt. \nonumber
\end{align}
By making the change of variables $\frac{y}{t} = u$, we obtain the representation
\begin{align}
    &\left(\int_{B(0,r)} \tau_x^\alpha |\mathcal{H}_\psi^\alpha f|^p(y) d\mu_\alpha(y)\right)^{\frac{1}{p}} \nonumber
    \\ 
    & \leq \int_{\R} \frac{|\psi(t)|}{|t|^{2\alpha + 2}} \left( \int_\R  |f(u)|^p \tau_{-x}^\alpha \chi_{B(0,r)}(tu)  |tu|^{2 \alpha + 1} |t|du \right)^{\frac{1}{p}} dt. \nonumber \\
    & = \int_{\R} \frac{|\psi(t)|}{|t|^{2\alpha + 2}} |t|^{\frac{2 \alpha + 2}{p}}\left( \int_\R  |f(u)|^p \tau_{-x}^\alpha \chi_{B(0,r)}(tu) d\mu_\alpha(u) \right)^{\frac{1}{p}} dt. \label{eq:3.1a3}
\end{align}
Now, we know that the support of $\tau_x^\alpha \chi_{B(0,r)}$ is contained in $B(x,r)$ i.e.,
$$
\tau_{-x}^\alpha \chi_{B(0,r)}(tu) \neq 0 
\quad \text{for} \quad 
tu \in B(-x,r),
$$
and using the definition of Dunkl ball, this implies
$$
\tau_{-x}^\alpha \chi_{B(0,r)}(tu) \neq 0 
\quad \text{for} \quad 
|tu| \in (max \{0, |x| -r\}, |x| + r ), 
$$
which further gives 
$$
\tau_{-x}^\alpha \chi_{B(0,r)}(tu) \neq 0 \text{ for } |u| \in \left(\max\left\{0, \frac{|x|-r}{|t|}\right\},\, \frac{|x|+r}{|t|}\right). 
$$
Hence
$$
\tau_{-x}^\alpha \chi_{B(0,r)}(tu) \neq 0 \text{ for } u \in B\!\left(\frac{x}{|t|}, \frac{r}{|t|}\right). 
$$
Thus,  we obtain from \eqref{eq:3.1a3}
\begin{align}
    &\left(\int_{B(0,r)} \tau_x^\alpha |\mathcal{H}_\psi^\alpha f|^p(y) d\mu_\alpha(y)\right)^{\frac{1}{p}} \nonumber
    \\ 
    &= \int_{\R} \frac{|\psi(t)|}{|t|^{2\alpha + 2}} |t|^{\frac{2 \alpha + 2}{p}}\left( \int_{B\left(\frac{x}{|t|}, \frac{r}{|t|}\right)}  |f(u)|^p \tau_{-x}^\alpha \chi_{B(0,r)}(tu) d\mu_\alpha(u) \right)^{\frac{1}{p}} dt \nonumber
    \\
    & \leq \int_{\R} \frac{|\psi(t)|}{|t|^{2\alpha + 2}}|t|^{\frac{2 \alpha + 2}{p}}\left( \int_{B\left(\frac{x}{|t|}, \frac{r}{|t|}\right)}  |f(u)|^p  d\mu_\alpha(u) \right)^{\frac{1}{p}} dt \nonumber
    \\
    & = \int_{\R} \frac{|\psi(t)|}{|t|^{2\alpha + 2}} |t|^{\frac{2 \alpha + 2}{p}}\left( \int_{I\left(-\frac{x}{|t|}, \frac{r}{|t|}\right) \cup I\left(\frac{x}{|t|}, \frac{r}{|t|}\right)}  |f(u)|^p  d\mu_\alpha(u) \right)^{\frac{1}{p}} dt \nonumber
    \\
    & \leq \int_{\R} \frac{|\psi(t)|}{|t|^{2\alpha + 2}} |t|^{\frac{2 \alpha + 2}{p}}\left( \int_{I\left(-\frac{x}{|t|}, \frac{r}{|t|}\right) }  |f(u)|^p  d\mu_\alpha(u) + \int_{I\left(\frac{x}{|t|}, \frac{r}{|t|}\right) }  |f(u)|^p  d\mu_\alpha(u) \right)^{\frac{1}{p}} dt,
    \label{eq:3.1a4}
\end{align}
since $B\left(\frac{x}{|t|}, \frac{r}{|t|} \right) = I\left(-\frac{x}{|t|}, \frac{r}{|t|} \right) \cup I\left(\frac{x}{|t|}, \frac{r}{|t|} \right)$, where $I\left(\frac{x}{|t|}, \frac{r}{|t|} \right) = \left ]\frac{x}{|t|} - \frac{r}{|t|}, \frac{x}{|t|} + \frac{r}{|t|} \right [$. Now using the following result (see \cite{spsa2}) 
\begin{align}
    \int_{I(y, \rho)} |f|^p(x) d\mu_\alpha(x) =\int_{B(0,\rho)} \tau_{y}^\alpha |f|^p(x) d\mu_\alpha(x), \nonumber
\end{align}
for all $y \in \R$ and for all $\rho > 0$,
we deduce from \eqref{eq:3.1a4}
\begin{align}
    &\left(\int_{B(0,r)} \tau_x^\alpha |\mathcal{H}_\psi^\alpha f|^p(y) d\mu_\alpha(y)\right)^{\frac{1}{p}} \nonumber
    \\ 
    &\leq \int_{\R} \frac{|\psi(t)|}{|t|^{2\alpha + 2}} |t|^{\frac{2 \alpha + 2}{p}} \left( \int_{B\!\left(0, \frac{r}{|t|}\right)}  \tau_{\frac{x}{|t|}}^\alpha |f|^p(u)  d\mu_\alpha(u) \right)^{\frac{1}{p}} dt  \nonumber
    \\
    & \leq  \int_\R \frac{|\psi(t)|}{|t|^{2\alpha + 2}} |t|^{\frac{2 \alpha + 2}{p}} \phi \left( \frac{r}{|t|} \right) \left( \frac{r}{|t|}\right)^{  \frac{d_\alpha}{p} } ||f||_{L^{p,\phi}(\R, d\mu_\alpha)} dt \nonumber
    \\
    & =  ||f||_{L^{p,\phi}(\R, d\mu_\alpha)} r^{ \frac{d_\alpha}{p} } \int_\R \frac{|\psi(t)|}{|t|^{d_\alpha}} \phi \left( \frac{r}{|t|} \right) dt. \label{eq:3.1a5}
\end{align}
Finally, we deduce from \eqref{eq:3.1a1} and \eqref{eq:3.1a5} that
\begin{align}
    ||\mathcal{H_\psi^\alpha}f||_{L^{p,\phi}(\R, d\mu_\alpha)} \leq  ||f||_{L^{p,\phi}(\R, d\mu_\alpha)}  \sup_{r > 0} \frac{1}{\phi(r)}\int_\R \frac{|\psi(t)|}{|t|^{d_\alpha}} \phi\left(\frac{r}{|t|} \right) dt, \nonumber
\end{align}
which completes the proof. \qed

As a special case of Theorem \ref{thm:3.1a}, we obtain the boundedness of Dunkl-type Hausdorff operators on Dunkl-type Morrey spaces, which is stated in the following theorem.
\begin{theorem} \label{thm:3.1}
Let $1 \leq p \leq q < \infty$. Then we have 
\begin{align}
    ||\mathcal{H}_\psi^\alpha f||_{L^{p,q}(\R, d\mu_\alpha)} \leq K_{\psi} ||f||_{L^{p,q}(\R, d\mu_\alpha)}, \nonumber
\end{align}
    where 
    \begin{align}
        K_{\psi} = \int_\R \frac{|\psi(t)|}{|t|^{d_\alpha\left(1 - \frac{1}{q}\right)}} dt < \infty. \nonumber
    \end{align}
\end{theorem}
\textbf{Proof.} The proof of Theorem \ref{thm:3.1} follows from Theorem \ref{thm:3.1a}, by taking $\phi(r)= r^{-\frac{d_\alpha}{q}}$. \qed

\begin{theorem} \label{thm:3.2a}
     Let $1 \leq p < \infty$ and $\phi$ be a positive measurable function on $\R^{+}$. Then the Dunkl-type Hausdorff operator is bounded from $\mathcal{L}^{p, \phi}(\R, d\mu_\alpha)$ to $\mathcal{L}^{p, \phi}(\R, d\mu_\alpha)$ i.e.
     \begin{align}
    ||\mathcal{H}_\psi^\alpha f||_{\mathcal{L}^{p,\phi}(\R, d\mu_\alpha)} \leq K_{\phi, \psi} ||f||_{\mathcal{L}^{p,\phi}(\R, d\mu_\alpha)}, \nonumber
\end{align}
    where $K_{\phi, \psi}$ is the same constant as in Theorem \ref{thm:3.1a}.
 \end{theorem}
\textbf{Proof.} From the definition of the generalized Dunkl-type Campanato norm together with the Minkowski inequality, we obtain:
\begin{align}
    ||\mathcal{H}_{\psi}^\alpha f ||_{\mathcal{L}^{p,\phi}(\R, d\mu_\alpha)} &= \sup_{\substack{{r>0}\\{x \in \R}}} \frac{1}{\phi(r)}\frac{1}{r^{\frac{d_\alpha}{p}}} \left(\int_{B(0,r)} |\tau_x^\alpha \mathcal{H}_\psi^\alpha f(y) - C_{B(0,r)}^{\psi, \alpha}(x)|^p d\mu_\alpha(y)\right)^{\frac{1}{p}}, \label{eq:3.1a6}
\end{align}
where $C_{B(0,r)}^{\psi, \alpha}(x) = \frac{1}{\mu_\alpha(B(0,r))} \int_{B(0,r)} \tau_x^\alpha \mathcal{H}_\psi^\alpha f(y) d\mu_\alpha(y)$.
Now, by using Lemma \ref{lem: 3.1a1}, we evaluate the integral term in \eqref{eq:3.1a6}
\begin{align} \small
    &\left(\int_{B(0,r)} |\tau_x^\alpha \mathcal{H}_\psi^\alpha f(y) - C_{B(0,r)}^{\psi, \alpha}(x)|^p d\mu_\alpha(y)\right)^{\frac{1}{p}} \nonumber
    \\
    &=\left(\int_{B(0,r)} \left|\tau_x^\alpha \mathcal{H}_\psi^\alpha f(y) - \frac{1}{\mu_\alpha(B(0,r))} \int_{B(0,r)} \tau_x^\alpha \mathcal{H}_\psi^\alpha f(z) d\mu_\alpha(z) \right|^p d\mu_\alpha(y)\right)^{\frac{1}{p}} \nonumber
    \\
    &=\left(\int_{B(0,r)} \left| \int_\R \frac{\psi(t)}{|t|^{2 \alpha + 2}} \tau_{x/t}^\alpha f\left( y/t \right) dt  - \frac{1}{\mu_\alpha(B(0,r))} \int_{B(0,r)}  \tau_x^\alpha \mathcal{H}_\psi^\alpha f(z) d\mu_\alpha(z) \right|^p d\mu_\alpha(y)\right)^{\frac{1}{p}}. \label{eq: 3.9}
    \end{align}
    Since by Fubini's theorem, 
    \begin{align}
        &\frac{1}{\mu_\alpha(B(0,r))} \int_{B(0,r)} \tau_x^\alpha \mathcal{H}_\psi^\alpha f(z) d\mu_\alpha(z) \nonumber
        \\
        &= \frac{1}{\mu_\alpha(B(0,r))} \int_{B(0,r)} \int_\R \frac{\psi(t)}{|t|^{2\alpha + 2}} \tau_{x/t}^\alpha f(z/t) dt d\mu_\alpha(z) \nonumber
        \\
        &= \int_\R \frac{\psi(t)}{|t|^{2\alpha + 2}} \frac{1}{\mu_\alpha(B(0,r))} \int_{B(0,r)}  \tau_{x/t}^\alpha f(z/t)  d\mu_\alpha(z) dt, \nonumber
    \end{align}
    for fixed $t \neq 0$, the change of variable $u = \frac{z}{t}$ gives
    \begin{align}
        &\frac{1}{\mu_\alpha(B(0,r))} \int_{B(0,r)}  \tau_x^\alpha \mathcal{H}_\psi^\alpha f(z) d\mu_\alpha(z) \nonumber
        \\
        &= \int_\R \frac{\psi(t)}{|t|^{2\alpha + 2}} \frac{1}{\mu_\alpha(B(0,r/t))} \int_{B(0,r/t)}  \tau_{x/t}^\alpha f(u)  d\mu_\alpha(u) dt, \nonumber 
        \\
        &= \int_\R \frac{\psi(t)}{|t|^{2\alpha + 2}} f^\alpha_{B(0,r/t)}(x/t) dt, \nonumber
    \end{align}
    where $\mu_\alpha(B(0,r/t)) = \frac{\mu_\alpha(B(0,r))}{|t|^{2\alpha + 2}}$ and $ f^\alpha_{B(0,r/t)}(x/t) = \frac{1}{\mu_\alpha(B(0,r/t))} \int_{B(0,r/t)}  \tau_{x/t}^\alpha f(u)  d\mu_\alpha(u)$. Hence \eqref{eq: 3.9} leads to 
    \begin{align}
         &\left(\int_{B(0,r)} |\tau_x^\alpha \mathcal{H}_\psi^\alpha f(y) - C_{B(0,r)}^{\psi, \alpha}(x)|^p d\mu_\alpha(y)\right)^{\frac{1}{p}} \nonumber
         \\
         &= \left(\int_{B(0,r)} \left|  \int_\R \frac{\psi(t)}{|t|^{2 \alpha + 2}} \tau_{x/t}^\alpha f\left( y/t \right) dt - \int_\R \frac{\psi(t)}{|t|^{2\alpha + 2}} f^\alpha_{B(0,r/t)}(x/t) dt \right|^p d\mu_\alpha(y)\right)^{\frac{1}{p}} \nonumber
         \\
         & \leq  \left(\int_{B(0,r)} \left( \int_\R \frac{|\psi(t)|}{|t|^{2\alpha + 2}} |\tau_{x/t}^\alpha f(y/t) -  f^\alpha_{B(0,r/t)}(x/t)| dt \right)^p d\mu_\alpha(y)\right)^{\frac{1}{p}} \nonumber
         \\
     & \leq  \int_\R \frac{|\psi(t)|}{|t|^{2\alpha + 2}} \left(\int_{B(0,r)} |\tau_{x/t}^\alpha f(y/t) -  f^\alpha_{B(0,r/t)}(x/t)|^p d\mu_\alpha(y) \right)^{\frac{1}{p}}  dt \nonumber,
\end{align}
    by applying the Minkowski inequality. Now
by making the change of variables $\frac{y}{t} = u$, we obtain 
\begin{align}
     &\left(\int_{B(0,r)} |\tau_x^\alpha \mathcal{H}_\psi^\alpha f(y) - C_{B(0,r)}^{\psi, \alpha}(x)|^p d\mu_\alpha(y)\right)^{\frac{1}{p}} \nonumber
    \\ 
    & \leq  \int_\R \frac{|\psi(t)|}{|t|^{2\alpha + 2}} \left(|t|^{2\alpha + 2}\int_{B(0,r/t)} |\tau_{x/t}^\alpha f(u) -  f^\alpha_{B(0,r/t)}(x/t)|^p d\mu_\alpha(u) \right)^{\frac{1}{p}}  dt \nonumber
    \\
    & \leq  \int_{\R} \frac{|\psi(t)|}{|t|^{2\alpha + 2}} |t|^{\frac{2 \alpha + 2}{p}} \phi \left( \frac{r}{|t|} \right) \left( \frac{r}{|t|}\right)^{ \frac{d_\alpha}{p} }  ||f||_{\mathcal{L}^{p,\phi}(\R, d\mu_\alpha)} dt \nonumber
      \\
    & =  ||f||_{\mathcal{L}^{p,\phi}(\R, d\mu_\alpha)}  r^{ \frac{d_\alpha}{p} } \int_\R \frac{|\psi(t)|}{|t|^{d_\alpha}} \phi \left( \frac{r}{|t|} \right) dt. \label{eq:3.1a7}
\end{align}
Thus, we deduce from \eqref{eq:3.1a6} and \eqref{eq:3.1a7} that
\begin{align}
    ||\mathcal{H_\psi^\alpha}f||_{\mathcal{L}^{p,\phi}(\R, d\mu_\alpha)} \leq  ||f||_{\mathcal{L}^{p,\phi}(\R, d\mu_\alpha)}  \sup_{r > 0} \frac{1}{\phi(r)}\int_\R \frac{|\psi(t)|}{|t|^{d_\alpha}} \phi\left(\frac{r}{|t|} \right) dt, \nonumber
\end{align}
which completes the proof. \qed

As a special case of Theorem \ref{thm:3.2a}, we obtain the boundedness of Dunkl-type Hausdorff operators on Dunkl-type Campanato spaces, which is stated in the following theorem.
\begin{theorem} \label{thm:3.2}
    Let $1 \leq p \leq q < \infty$. Then, for the same constant $K_\psi$  as in Theorem \ref{thm:3.1}, we have 
    \begin{align}
        ||\mathcal{H}_\psi^\alpha f||_{\mathcal{L}^{p,q}(\R, d\mu_\alpha)} \leq K_\psi ||f||_{\mathcal{L}^{p,q}(\R, d\mu_\alpha)}. \nonumber
    \end{align}
\end{theorem} 
\textbf{Proof.} Again the proof of Theorem \ref{thm:3.2} follows from Theorem \ref{thm:3.2a}, by taking $\phi(r)= r^{-\frac{d_\alpha}{q}}$. \qed
\section{\bf{ Fractional Hausdorff operators associated with Dunkl operator on the real line}} \label{sec:4}
In this section, we will introduce the concept of the fractional Hausdorff operator in the Dunkl setting, which we refer to as the fractional Dunkl-type Hausdorff operator. Towards this, first we recall the definition of the Dunkl-type Hausdorff operator defined in \eqref{eq:3.1a}. By making the change of variable $z = \frac{x}{t}$ in \eqref{eq:3.1a}, we observe that the Dunkl-type Hausdorff operator can equivalently be written as 
\begin{align}
    \mathcal{H}_{\psi}^\alpha f(x) :=  \frac{1}{|x|^{2\alpha + 1}} \int_\R \frac{\psi(xz^{-1})}{|z|} f (z)  d\mu_\alpha(z), ~ x \in \R - \{0\}.\label{eq:4.1}
\end{align}
Motivated by this, we provide the definition of the fractional Dunkl-type Hausdorff operator as follows:
\begin{definition}
For a function $\psi \in L^1(\R)$ and for $0 \leq \beta < d_\alpha$, the fractional Dunkl-type Hausdorff operator, denoted by $\mathcal{H}_{\psi,\beta}^\alpha$, is defined as 
\begin{align}
    \mathcal{H}_{\psi,\beta}^\alpha f(x) :=  \frac{1}{|x|^{2\alpha + 1}} \int_\R \frac{\psi(xz^{-1})}{|z|^{1-\beta}} f (z)  d\mu_\alpha(z), ~ x \in \R - \{0\}. \label{eq:4.1a}
\end{align}
\end{definition}
Now we observe that when $\beta =0$, then the fractional Dunkl-type Hausdorff operator $\mathcal{H}_{\psi,\beta}^\alpha$ reduces to the Dunkl-type Hausdorff operator which is defined in \eqref{eq:4.1}.    
Further, for $\alpha= -{1/2}$, the fractional Dunkl-type Hausdorff operator $\mathcal{H}_{\psi, \beta}^\alpha$ becomes the classical fractional Hausdorff operator (see \eqref{eq: 1.2}).

The goal of this section is to prove the boundedness of $\mathcal{H}_{\psi,\beta}^\alpha$ on Dunkl-type Lebesgue spaces, Dunkl-type Morrey spaces and generalized Dunkl-type Morrey spaces, respectively. Before providing proofs of these results, we first recall the following lemma (Young's inequality for convolution). 
\begin{lemma} \cite{grafakos} \label{lem:4.1}
Let $1 \leq p, q, r \leq \infty$ satisfy 
$\frac{1}{q} = \frac{1}{p} + \frac{1}{r} - 1$
and let $\nu$ be a Haar measure on a locally compact group $G$. Then
\begin{align}
\|f * g\|_{L^q(G,\nu)} \leq \|g\|_{L^r(G,\nu)} \, \|f\|_{L^p(G,\nu)}, \nonumber
\end{align}
for all $f \in L^p(G,\nu)$ and $g \in L^r(G,\nu)$ such that 
$\|g\|_{L^r(G,\nu)} = \|\tilde{g}\|_{L^r(G,\nu)}$, 
where $\tilde{g}(x) = g(x^{-1})$.
\end{lemma}
\begin{theorem} \label{4.1}
    Let $\psi$ be a measurable function on $\R$ such that  
    \begin{align}
       C_{\psi, s, q, \alpha} = \left(\int_{\R} |\psi(z)|^s |z|^{s\left(\frac{d_\alpha}{q} - (2\alpha +1) \right) -1} dz \right)^{1/s} < \infty \label{eq:3.9},
    \end{align}
   where $ 1 \leq p, q \leq \infty$, $0 \leq \beta < d_\alpha$ with $\frac{1}{q} = \frac{1}{p} - \frac{\beta}{d_\alpha}$ and $s = \frac{d_\alpha}{d_\alpha - \beta}$ . Then $\mathcal{H}_{\psi, \beta}^\alpha$ is bounded from $L^{p}(\R, d\mu_\alpha)$ to $L^{q}(\R, d\mu_\alpha)$.
\end{theorem}

\textbf{Proof.} We know that the multiplicative group $\R^*= \R - \{0\}$ is a locally compact group with Haar measure $\frac{dx}{|x|}$. Now using the difinition of $\mathcal{H}_{\psi,\beta}^\alpha$ (see \eqref{eq:4.1a}), we have
\begin{align}
    ||\mathcal{H}_{\psi, \beta}^\alpha f ||_{L^q(\R, d\mu_\alpha)} &= \left( \int_\R |\mathcal{H}_{\psi, \beta}^\alpha f(x)|^q d\mu_\alpha(x) \right)^{1/q} \nonumber
    \\
    & = \left( \int_\R \left| \frac{1}{|x|^{2\alpha + 1}} \int_\R \frac{\psi(xz^{-1})}{|z|^{1-\beta}} f(z) |z|^{2\alpha + 1} dz \right|^q |x|^{2\alpha + 1} dx \right)^{1/q} \nonumber
    \\
    & = \left( \int_\R \left|  \int_\R \psi(xz^{-1}) |x|^{\frac{2\alpha + 2}{q} - (2 \alpha + 1)} |z|^{\beta + 2 \alpha + 1} f(z)  \frac{dz}{|z|} \right|^q  \frac{dx}{|x|} \right)^{1/q} \nonumber
    \\
     & \leq \left( \int_{\R^*} \left(  \int_{\R^*} |\psi(xz^{-1})| |xz^{-1}|^{\frac{2\alpha + 2}{q} - (2 \alpha + 1)} |z|^{\frac{2 \alpha + 2}{p}} |f(z)|  \frac{dz}{|z|} \right)^q  \frac{dx}{|x|} \right)^{1/q}, \nonumber
\end{align}
where we have used the relation $\frac{1}{q} = \frac{1}{p} - \frac{\beta}{d_\alpha}$. The expression on the right-hand side may be viewed as a convolution on the multiplicative group $\R^*$ endowed with the Haar measure $\frac{dx}{|x|}$ i.e.
\begin{align}
     ||\mathcal{H}_{\psi, \beta}^\alpha f ||_{L^q(\R, d\mu_\alpha)} &\leq ||F * G ||_{L^q(\R^*, \frac{dx}{|x|})}, \nonumber
\end{align}
where $G(z):= \psi(z) |z|^{\frac{2\alpha +2}{q} -(2\alpha + 1)}$ and $F(z)= f(z) |z|^{\frac{2 \alpha + 2}{p}}$. 
This allows us to apply Lemma \ref{lem:4.1} with $\frac{1}{q} = \frac{1}{p} + \frac{1}{s} - 1$, yielding
\begin{align}
     ||\mathcal{H}_{\psi, \beta}^\alpha f ||_{L^q(\R, d\mu_\alpha)} & \leq ||F||_{L^p(\R^*, \frac{dz}{|z|})} ||G||_{L^s(\R^*, \frac{dz}{|z|})} \nonumber
     \\
     &= \left( \int_{\R^*} |\psi(z)|^s |z|^{s\left(\frac{2\alpha + 2}{q} - (2 \alpha + 1)\right)-1} dz \right)^{1/s} \left( \int_{\R^*} |f(z)|^p |z|^{2 \alpha + 1} dz \right)^{1/p}.\nonumber
\end{align}
Thus, from \eqref{eq:3.9}, we have
\begin{align}
    ||\mathcal{H}_{\psi, \beta}^\alpha f ||_{L^q(\R, d\mu_\alpha)} \leq C_{\psi, s, q, \alpha} ||f||_{L^{p}(\R, d\mu_\alpha)}. \nonumber
\end{align} \qed

\begin{theorem} \label{4.3}
    Let $\psi$ be a measurable function on $\R$ and $0 < \beta < d_\alpha$. If $\phi: \R^{+} \to \R^{+}$ satisfies the doubling condition such that $\int_r^\infty \phi(t) t^{\frac{d_\alpha}{m'} -1} dt \leq C \phi(r) r^{\frac{d_\alpha}{m'}}$ for every $r>0$ and 
    \begin{align}
        \frac{|\psi(xz^{-1})|}{|x|^{2\alpha + 1} |z|^{1-\beta}} \leq \min\{ |z|^{\beta - d_\alpha}, |x|^{\beta - d_\alpha}\} \label{eq: 4.4}
   \end{align}
   for all $x,z \neq 0$, then $\mathcal{H}_{\psi, \beta}^\alpha$ is bounded from $L^{p_1, \phi}(\R, d\mu_\alpha)$ to $L^{p_2, \omega}(\R, d\mu_\alpha)$ with
        \begin{align}
    ||\mathcal{H}_{\psi,\beta}^\alpha f||_{L^{p_2, \omega}(\R, d\mu_\alpha)} \leq ||f||_{L^{p_1, \phi}(\R, d\mu_\alpha)}  \left\Vert \frac{1}{|.|^{d_\alpha - \beta}} \right\Vert_{L^{l,m}(\R, d\mu_\alpha)},\nonumber
\end{align}
 where $1 \leq p_1 < \frac{d_\alpha}{\beta}$, $\frac{1}{p_2} = \frac{1}{p_1} - \frac{1}{l'}$ and $\omega(r)= \phi(r) r^{\frac{d_\alpha}{m'}}$ with $1 \leq l < m= \frac{d_\alpha}{d_\alpha -\beta}$.
\end{theorem}

\textbf{Proof.} Let $f \in L^{p_1, \phi}(\R, d\mu_\alpha)$ and let $B = B(0,r)$ for some $r > 0$. We decompose $f$ as $f = f_1 + f_2 = f\chi_B + f\chi_{B^c}$, where $B^c$ denotes the complement of $B$. We first consider the integral
\begin{align}
     \mathcal{H}_{\psi,\beta}^\alpha f_1 (x) &= \frac{1}{|x|^{2 \alpha + 1}} \int_R \frac{\psi(xz^{-1})}{|z|^{1-\beta}} f_1(z) d\mu_\alpha(z) \nonumber
     \\
      & =\int_{B(0,r)} f(z) \frac{\psi(xz^{-1})}{|x|^{2\alpha + 1}|z|^{1-\beta}} d\mu_\alpha(z). \nonumber
\end{align} 
Then using the H\"older's inequality, we obtain
\begin{align}
    &| \mathcal{H}_{\psi,\beta}^\alpha f_1 (x)| \nonumber
    \\
    & \leq \int_{B(0,r)} |f(z)| \frac{|\psi(xz^{-1})|}{|x|^{2\alpha + 1}|z|^{1-\beta}} d\mu_\alpha(z)\nonumber
   \\
    &= \int_{B(0,r)} |f(z)|^{\frac{p_1}{p_2}} \left( \frac{|\psi(xz^{-1})|}{|x|^{2\alpha + 1}|z|^{1-\beta}} \right)^{\frac{l}{p_2}} |f(z)|^{\frac{p_2-p_1}{p_2}} \left( \frac{|\psi(xz^{-1})|}{|x|^{2\alpha + 1}|z|^{1-\beta}} \right)^{\frac{p_2 -l}{p_2}} d\mu_\alpha(z)\nonumber
   \end{align}
   \begin{align}
    &\leq \left( \int_{B(0,r)} \left( \frac{|\psi(xz^{-1})|}{|x|^{2\alpha + 1}|z|^{1-\beta}} \right)^l |f(z)|^{p_1} d\mu_\alpha(z) \right)^{\frac{1}{p_2}} \nonumber
  \\
    & \hspace{0.5cm} \times \left( \int_{B(0,r)} \left( \frac{|\psi(xz^{-1})|}{|x|^{2\alpha + 1}|z|^{1-\beta}} \right)^{\frac{p_2 - l}{p_2} p'_2} |f(z)|^{\frac{p_2 - p_1}{p_2} p'_2} d\mu_\alpha(z) \right)^{\frac{1}{p'_2}} \nonumber
  \\
     &\leq \left( \int_{B(0,r)} \left( \frac{|\psi(xz^{-1})|}{|x|^{2\alpha + 1}|z|^{1-\beta}} \right)^l |f(z)|^{p_1} d\mu_\alpha(z) \right)^{\frac{1}{p_2}} \nonumber
    \\
    & \hspace{0.5cm} \times \left( \int_{B(0,r)} \left( \frac{|\psi(xz^{-1})|}{|x|^{2\alpha + 1}|z|^{1-\beta}} \right)^{\frac{p_2 -l}{p_2 -1}} |f(z)|^{\frac{p_2 - p_1}{p_2 -1}} d\mu_\alpha(z) \right)^{\frac{1}{p'_2}}. \label{eq:4.5}
\end{align}
Again using the H\"older's inequality with $p=\frac{l'}{p'_2}$ and $q= \frac{l(p_2 -1)}{p_2 -l}$ in the second integral of \eqref{eq:4.5} and observing facts $\frac{p_2 -l}{l(p_2 -1)} = p'_2 \left( \frac{1}{l} - \frac{1}{p_2} \right)$ and $\frac{p_2 - p_1 }{p_2 -1}\cdot \frac{l'}{p'_2} = p_1$, we can rewrite the second integral as
\begin{align}
    &\int_{B(0,r)} \left( \frac{|\psi(xz^{-1})|}{|x|^{2\alpha + 1}|z|^{1-\beta}} \right)^{\frac{p_2 -l}{p_2 -1}} |f(z)|^{\frac{p_2 - p_1}{p_2 -1}} d\mu_\alpha(z) \nonumber
    \\
    & \leq \left( \int_{B(0,r)} |f(z)|^{\frac{p_2 - p_1}{p_2 -1}. \frac{l'}{p'_2}} d\mu_\alpha(z) \right)^{\frac{p'_2}{l'}}\nonumber
    \\
    & \hspace{0.5 cm} \times \left( \int_{B(0,r)} \left( \frac{|\psi(xz^{-1})|}{|x|^{2\alpha + 1}|z|^{1-\beta}} \right)^{\frac{p_2 -l}{p_2 -1} \frac{l(p_2 -1)}{p_2 -l}}  d\mu_\alpha(z) \right)^{p'_2 \left( \frac{1}{l} - \frac{1}{p_2} \right)} \nonumber
   \\
    & \leq \left( \int_{B(0,r)} |f(z)|^{p_1} d\mu_\alpha(z) \right)^{\frac{p'_2}{l'}}  \left( \int_{B(0,r)} \left( \frac{|\psi(xz^{-1})|}{|x|^{2\alpha + 1}|z|^{1-\beta}} \right)^l  d\mu_\alpha(z) \right)^{p'_2 \left( \frac{1}{l} - \frac{1}{p_2} \right)}. \nonumber
\end{align}
Thus using \eqref{eq: 4.4}, we get from \eqref{eq:4.5}
\begin{align}
   &| \mathcal{H}_{\psi,\beta}^\alpha f_1 (x)| \nonumber
   \\
   & \leq \left( \int_{B(0,r)} \left( \frac{|\psi(xz^{-1})|}{|x|^{2\alpha + 1}|z|^{1-\beta}} \right)^l |f(z)|^{p_1} d\mu_\alpha(z) \right)^{\frac{1}{p_2}} \left( \int_{B(0,r)} |f(z)|^{p_1} d\mu_\alpha(z) \right)^{\frac{1}{l'}} \nonumber
   \\
   & \hspace{1 cm}\times \left( \int_{B(0,r)} \left( \frac{|\psi(xz^{-1})|}{|x|^{2\alpha + 1}|z|^{1-\beta}} \right)^l  d\mu_\alpha(z) \right)^{ \frac{1}{l} - \frac{1}{p_2} } \nonumber
   \\
    & \leq \left( \int_{B(0,r)} \left( \frac{|\psi(xz^{-1})|}{|x|^{2\alpha + 1}|z|^{1-\beta}} \right)^l |f(z)|^{p_1} d\mu_\alpha(z) \right)^{\frac{1}{p_2}} \left( \int_{B(0,r)} |f(z)|^{p_1} d\mu_\alpha(z) \right)^{\frac{1}{l'}} \nonumber
   \\
   & \hspace{1 cm}\times \left( \int_{B(0,r)} \left( \frac{1}{|z|^{d_\alpha - \beta}} \right)^l  d\mu_\alpha(z) \right)^{ \frac{1}{l} - \frac{1}{p_2} } \nonumber
   \\
   & \leq  \left( \int_{B(0,r)} \left( \frac{|\psi(xz^{-1})|}{|x|^{2\alpha + 1}|z|^{1-\beta}} \right)^l |f(z)|^{p_1} d\mu_\alpha(z) \right)^{\frac{1}{p_2}} r^{ \frac{d_\alpha}{l'} } \phi^{\frac{p_1}{l'}}(r) ||f||_{L^{p_1, \phi}(\R, d\mu_\alpha)}^{\frac{p_1}{l'}}
   \nonumber      
 \end{align}
   \begin{align}
   & \hspace{.5cm}\times r^{d_\alpha \left( \frac{1}{l} - \frac{1}{m} \right) . \left( 1- \frac{l}{p_2} \right)} \left\Vert \frac{1}{|.|^{d_\alpha - \beta}} \right\Vert_{L^{l,m}(\R, d\mu_\alpha)}^{1- \frac{l}{p_2}} \nonumber
   \\
   & \leq r^{ d_\alpha \left( 1 - \frac{l}{m} \right) \left( \frac{1}{l} - \frac{1}{p_2} \right) + \frac{d_\alpha}{l'}  } \phi^{\frac{p_1}{l'}} (r)||f||_{L^{p_1, \phi}(\R, d\mu_\alpha)}^{\frac{p_1}{l'}} \left\Vert \frac{1}{|.|^{d_\alpha - \beta}} \right\Vert_{L^{l,m}(\R, d\mu_\alpha)}^{1- \frac{l}{p_2}}  \nonumber
   \\
   & \hspace{1cm} \times  \left( \int_{B(0,r)} \left( \frac{|\psi(xz^{-1})|}{|x|^{2\alpha + 1}|z|^{1-\beta}} \right)^l |f(z)|^{p_1} d\mu_\alpha(z) \right)^{\frac{1}{p_2}}. \nonumber
\end{align}
Let $A(x)= \displaystyle \int_{B(0,r)} \left( \frac{|\psi(xz^{-1})|}{|x|^{2\alpha + 1}|z|^{1-\beta}} \right)^l |f(z)|^{p_1} d\mu_\alpha(z) $. Now for $x_0 \in \R$, we consider 
\begin{align}
    &\int_{B(0,r)} \tau_{x_0}^\alpha | \mathcal{H}_{\psi,\beta}^\alpha f_1|^{p_2}(x) d\mu_\alpha(x) \nonumber
   \\
    &= \int_\R  | \mathcal{H}_{\psi,\beta}^\alpha f_1|^{p_2}(x) \tau_{-x_0}^\alpha \chi_{B(0,r)}(x) d\mu_\alpha(x) 
    \nonumber
    \\
    & \leq  r^{ d_\alpha p_2 \left( 1 - \frac{l}{m} \right) \left( \frac{1}{l} - \frac{1}{p_2} \right) + \frac{d_\alpha p_2}{l'} } \phi^{\frac{p_1 p_2}{l'}}(r) ||f||_{L^{p_1, \phi}(\R, d\mu_\alpha)}^{\frac{p_1 p_2}{l'}} \left\Vert \frac{1}{|.|^{d_\alpha - \beta}} \right\Vert_{L^{l,m}(\R, d\mu_\alpha)}^{p_2- l}  \nonumber
   \\
   & \hspace{1cm} \times   \int_\R  A(x) \tau_{-x_0}^\alpha \chi_{B(0,r)}(x) d\mu_\alpha(x) 
   \nonumber
  \\
   &  \leq  r^{ d_\alpha p_2 \left( 1 - \frac{l}{m} \right) \left( \frac{1}{l} - \frac{1}{p_2} \right) + \frac{d_\alpha p_2}{l'}  } \phi^{\frac{p_1 p_2}{l'}}(r) ||f||_{L^{p_1, \phi}(\R, d\mu_\alpha)}^{\frac{p_1 p_2}{l'}} \left\Vert \frac{1}{|.|^{d_\alpha - \beta}} \right\Vert_{L^{l,m}(\R, d\mu_\alpha)}^{p_2- l}  \nonumber
   \\
   & \hspace{1cm} \times   \int_\R  \int_{B(0,r)} \left( \frac{|\psi(xz^{-1})|}{|x|^{2\alpha + 1}|z|^{1-\beta}} \right)^l |f(z)|^{p_1} d\mu_\alpha(z) \tau_{-x_0}^\alpha \chi_{B(0,r)}(x) d\mu_\alpha(x) .
   \nonumber
\end{align}
Then applying Fubini's theorem and our assumption \eqref{eq: 4.4}, we obtain 
\begin{align}
     &\int_{B(0,r)} \tau_{x_0}^\alpha | \mathcal{H}_{\psi,\beta}^\alpha f_1|^{p_2}(x) d\mu_\alpha(x) \nonumber
    \\ 
    &  \leq  r^{ d_\alpha p_2 \left( 1 - \frac{l}{m} \right) \left( \frac{1}{l} - \frac{1}{p_2} \right) + \frac{d_\alpha p_2}{l'}  } \phi^{\frac{p_1 p_2}{l'}}(r)||f||_{L^{p_1, \phi}(\R, d\mu_\alpha)}^{\frac{p_1 p_2}{l'}} \left\Vert \frac{1}{|.|^{d_\alpha - \beta}} \right\Vert_{L^{l,m}(\R, d\mu_\alpha)}^{p_2- l}  \nonumber
\\
   & \hspace{1cm} \times  \int_{B(0,r)} |f(z)|^{p_1} \left(\int_\R   \left( \frac{|\psi(xz^{-1})|}{|x|^{2\alpha + 1}|z|^{1-\beta}} \right)^l  \tau_{-x_0}^\alpha \chi_{B(0,r)}(x) d\mu_\alpha(x) \right)  d\mu_\alpha(z) 
   \nonumber
  \\
   &  \leq r^{ d_\alpha p_2 \left( 1 - \frac{l}{m} \right) \left( \frac{1}{l} - \frac{1}{p_2} \right) + \frac{d_\alpha p_2}{l'}  } \phi^{\frac{p_1 p_2}{l'}}(r) ||f||_{L^{p_1, \phi}(\R, d\mu_\alpha)}^{\frac{p_1 p_2}{l'}} \left\Vert \frac{1}{|.|^{d_\alpha - \beta}} \right\Vert_{L^{l,m}(\R, d\mu_\alpha)}^{p_2- l}  \nonumber
   \\
   & \hspace{1cm} \times  \int_{B(0,r)} |f(z)|^{p_1} \left(\int_\R   \left( \frac{1}{|x|^{d_\alpha - \beta}} \right)^l  \tau_{-x_0}^\alpha \chi_{B(0,r)}(x) d\mu_\alpha(x) \right)  d\mu_\alpha(z) 
   \nonumber
   \\
   & =  r^{ d_\alpha p_2 \left( 1 - \frac{l}{m} \right) \left( \frac{1}{l} - \frac{1}{p_2} \right) + \frac{d_\alpha p_2}{l'}  } \phi^{\frac{p_1 p_2}{l'}}(r) ||f||_{L^{p_1, \phi}(\R, d\mu_\alpha)}^{\frac{p_1 p_2}{l'}} \left\Vert \frac{1}{|.|^{d_\alpha - \beta}} \right\Vert_{L^{l,m}(\R, d\mu_\alpha)}^{p_2- l}  \nonumber
   \\
   & \hspace{1cm} \times  \int_{B(0,r)} |f(z)|^{p_1} \left(\int_{B(0,r)}   \tau_{x_0}^\alpha \left( \frac{1}{|.|^{d_\alpha - \beta}} \right)^l(x) d\mu_\alpha(x) \right)  d\mu_\alpha(z) 
   \nonumber
  \end{align}
   \begin{align}
   & =  r^{ d_\alpha p_2 \left( 1 - \frac{l}{m} \right) \left( \frac{1}{l} - \frac{1}{p_2} \right) + \frac{d_\alpha p_2}{l'}  } \phi^{\frac{p_1 p_2}{l'}}(r) ||f||_{L^{p_1, \phi}(\R, d\mu_\alpha)}^{\frac{p_1 p_2}{l'}} \left\Vert \frac{1}{|.|^{d_\alpha - \beta}} \right\Vert_{L^{l,m}(\R, d\mu_\alpha)}^{p_2- l}  \nonumber
   \\
   & \hspace{1cm} \times  \int_{B(0,r)} |f(z)|^{p_1}   d\mu_\alpha(z) \int_{B(0,r)}   \tau_{x_0}^\alpha \left( \frac{1}{|.|^{d_\alpha - \beta}} \right)^l(x) d\mu_\alpha(x) 
   \nonumber
  \\
   &  \leq  r^{ d_\alpha p_2 \left( 1 - \frac{l}{m} \right) \left( \frac{1}{l} - \frac{1}{p_2} \right) + \frac{d_\alpha p_2}{l'}  } \phi^{\frac{p_1 p_2}{l'}}(r) ||f||_{L^{p_1, \phi}(\R, d\mu_\alpha)}^{\frac{p_1 p_2}{l'}} \left\Vert \frac{1}{|.|^{d_\alpha - \beta}} \right\Vert_{L^{l,m}(\R, d\mu_\alpha)}^{p_2- l}  \nonumber
   \\
   & \hspace{1cm} \times  r^{d_\alpha } \phi^{p_1}(r) ||f||_{L^{p_1, \phi}(\R, d\mu_\alpha)}^{p_1} r^{d_\alpha \left( 1 - \frac{l}{m} \right)}  \left\Vert \frac{1}{|.|^{d_\alpha - \beta}} \right\Vert_{L^{l,m}(\R, d\mu_\alpha)}^l
   \nonumber
   \\
   & = r^{d_\alpha\left( 1 + \frac{p_2}{m'} \right)} \phi^{p_2}(r)   \left\Vert \frac{1}{|.|^{d_\alpha - \beta}} \right\Vert_{L^{l,m}(\R, d\mu_\alpha)}^{p_2} ||f||_{L^{p_1, \phi}(\R, d\mu_\alpha)}^{p_2}  \nonumber
    \\
   & = r^{d_\alpha}  \omega^{p_2}(r)   \left\Vert \frac{1}{|.|^{d_\alpha - \beta}} \right\Vert_{L^{l,m}(\R, d\mu_\alpha)}^{p_2} ||f||_{L^{p_1, \phi}(\R, d\mu_\alpha)}^{p_2},  \nonumber
\end{align}
where $\omega(r)= \phi(r) r^{\frac{d_\alpha}{m'}}$ and we have used relation $\frac{1}{p_2} = \frac{1}{p_1} - \frac{1}{l'}$. Thus, we arrive at
\begin{align}
     &  \frac{1}{\omega(r)}  \left( \frac{1}{r^{ d_\alpha}}\int_{B(0,r)} \tau_{x_0}^\alpha | \mathcal{H}_{\psi,\beta}^\alpha f_1|^{p_2}(x) d\mu_\alpha(x) \right)^{\frac{1}{p_2}} \nonumber
     \\
     &\leq   \left\Vert \frac{1}{|.|^{d_\alpha - \beta}} \right\Vert_{L^{l,m}(\R, d\mu_\alpha)} ||f||_{L^{p_1, \phi}(\R, d\mu_\alpha)}. \nonumber
\end{align}
Finally taking the supremum over $r> 0$ and $x_0 \in \R$ and using the definition of Morrey norm associated with the Dunkl operator on the real line, we get 
\begin{align}
    || \mathcal{H}_{\psi,\beta}^\alpha f_1||_{L^{p_2, \omega}(\R, d\mu_\alpha)} \leq  \left\Vert \frac{1}{|.|^{d_\alpha - \beta}} \right\Vert_{L^{l,m}(\R, d\mu_\alpha)} ||f||_{L^{p_1, \phi}(\R, d\mu_\alpha)}. \label{eq: 4.8a}
\end{align} 
Now we consider the integral 
\begin{align}
    \mathcal{H}_{\psi,\beta}^\alpha f_2(x) &= \frac{1}{|x|^{2 \alpha + 1}} \int_R \frac{\psi(xz^{-1})}{|z|^{1-\beta}} f_2(z) d\mu_\alpha(z) \nonumber
     \\
      & =\int_{B^c(0,r)} f(z) \frac{\psi(xz^{-1})}{|x|^{2\alpha + 1}|z|^{1-\beta}} d\mu_\alpha(z). \nonumber
\end{align}
 By using dyadic decomposition and \eqref{eq: 4.4}, we obtain
\begin{align}
    | \mathcal{H}_{\psi,\beta}^\alpha f_2 (x)| & \leq \int_{B^c(0,r)} |f(z)| \frac{|\psi(xz^{-1})|}{|x|^{2\alpha + 1}|z|^{1-\beta}} d\mu_\alpha(z)\nonumber
\\
    & = \sum_{k=0}^\infty \int_{2^k r \leq |y| < 2^{k+1} r} |f(z)| \frac{|\psi(xz^{-1})|}{|x|^{2\alpha + 1}|z|^{1-\beta}} d\mu_\alpha(z)\nonumber 
    \\  
    & \leq \sum_{k=0}^\infty \int_{2^k r \leq |y| < 2^{k+1} r} |f(z)| \frac{1}{|z|^{d_\alpha -\beta}} d\mu_\alpha(z)\nonumber 
   \end{align}
   \begin{align}
    & \leq \sum_{k=0}^\infty \frac{1}{(2^k r)^{d_\alpha -\beta}} \int_{2^k r \leq |y| < 2^{k+1} r} |f(z)|  d\mu_\alpha(z)\nonumber
    \\
    & \leq \sum_{k=0}^\infty \frac{1}{(2^k r)^{d_\alpha -\beta}} \left( \int_{0 \leq |y| < 2^{k+1} r} |f(z)|^{p_1}  d\mu_\alpha(z) \right)^{\frac{1}{p_1}} (2^{k+1} r)^{\frac{d_\alpha}{p'_1}}\nonumber
     \\  
    & \leq  ||f||_{L^{p_1, \phi}(\R, d\mu_\alpha)} \sum_{k=0}^\infty \frac{1}{(2^k r)^{d_\alpha -\beta}}   (2^{k+1} r)^{d_\alpha} \phi(2^{k+1}r).\label{eq: 4.9a}
\end{align}
As we observe that
\begin{align}
    \frac{1}{(2^k r)^{d_\alpha - \beta}} &\leq (2^k r)^{-\frac{d_\alpha}{l}} \left( \int_{0 \leq |z| < 2^{k+1} r} \frac{1}{|z|^{(d_\alpha -\beta)l}} d\mu_\alpha(y) \right )^{\frac{1}{l}} \nonumber
    \\
    &\leq (2^k r)^{-\frac{d_\alpha}{m}}  \left\Vert \frac{1}{|.|^{d_\alpha - \beta}} \right\Vert_{L^{l,m}(\R, d\mu_\alpha)}, 
\end{align}
it follows from \eqref{eq: 4.9a}
\begin{align}
      | \mathcal{H}_{\psi,\beta}^\alpha f_2 (x)| &\leq  ||f||_{L^{p_1, \phi}(\R, d\mu_\alpha)}  \left\Vert \frac{1}{|.|^{d_\alpha - \beta}} \right\Vert_{L^{l,m}(\R, d\mu_\alpha)} \sum_{k=0}^\infty    (2^{k+1} r)^{d_\alpha \left(1 - \frac{1}{m}\right)}  \phi(2^{k+1} r) \nonumber
      \\
      &\leq  ||f||_{L^{p_1, \phi}(\R, d\mu_\alpha)}  \left\Vert \frac{1}{|.|^{d_\alpha - \beta}} \right\Vert_{L^{l,m}(\R, d\mu_\alpha)} \sum_{k=0}^\infty    (2^{k+1} r)^{\frac{d_\alpha}{m'} } \phi(2^{k+1} r) \nonumber
      \\
      &\leq  ||f||_{L^{p_1, \phi}(\R, d\mu_\alpha)}  \left\Vert \frac{1}{|.|^{d_\alpha - \beta}} \right\Vert_{L^{l,m}(\R, d\mu_\alpha)} \sum_{k=0}^\infty   \int_{2^k r}^{2^{k+1}r} \phi(t) t^{\frac{d_\alpha}{m'}-1} dt \nonumber
      \\
       &\leq C ||f||_{L^{p_1, \phi}(\R, d\mu_\alpha)}  \left\Vert \frac{1}{|.|^{d_\alpha - \beta}} \right\Vert_{L^{l,m}(\R, d\mu_\alpha)}  \int_r^\infty \phi(t) t^{\frac{d_\alpha}{m'} -1} dt\nonumber
       \\
       & \leq C ||f||_{L^{p_1, \phi}(\R, d\mu_\alpha)}  \left\Vert \frac{1}{|.|^{d_\alpha - \beta}} \right\Vert_{L^{l,m}(\R, d\mu_\alpha)}  \phi(r) r^{\frac{d_\alpha}{m'}}\nonumber
       \\
       & = C ||f||_{L^{p_1, \phi}(\R, d\mu_\alpha)}  \left\Vert \frac{1}{|.|^{d_\alpha - \beta}} \right\Vert_{L^{l,m}(\R, d\mu_\alpha)}  \omega(r), \label{eq:4.10}
\end{align}
where we have used the assumptions in the statement of Theorem \ref {4.3} and   
\begin{align}
    \int_{2^k r}^{2^{k+1}r} \phi(t) t^{\frac{d_\alpha}{m'}-1} dt &\geq C \phi(2^{k+1}r) \int_{2^k r}^{2^{k+1}r} t^{\frac{d_\alpha}{m'} -1} dt \nonumber
    \\
    &\geq C\phi(2^{k+1} r) (2^{k+1} r)^{\frac{d_\alpha}{m'}}  . \nonumber
\end{align}
Thus from \eqref{eq:4.10}
\begin{align}
   & \left( \int_{B(0,r)} \tau_{x_0}^\alpha   | \mathcal{H}_{\psi,\beta}^\alpha f_2 (x)|^{p_2} d\mu_\alpha(x) \right)^{\frac{1}{p_2}} \nonumber
   \\
   & =\left( \int_{\R}    | \mathcal{H}_{\psi,\beta}^\alpha f_2 (x)|^{p_2} \tau_{-x_0}^\alpha \chi_{B(0,r)}(x) d\mu_\alpha(x) \right)^{\frac{1}{p_2}} \nonumber
   \\
   &\leq C||f||_{L^{p_1, \phi}(\R, d\mu_\alpha)}  \left\Vert \frac{1}{|.|^{d_\alpha - \beta}} \right\Vert_{L^{l,m}(\R, d\mu_\alpha)}  \omega(r)  \left( \int_{\R}    \tau_{-x_0}^\alpha \chi_{B(0,r)}(x) d\mu_\alpha(x) \right)^{\frac{1}{p_2}} \nonumber
   \\
   & \leq C ||f||_{L^{p_1, \phi}(\R, d\mu_\alpha)}  \left\Vert \frac{1}{|.|^{d_\alpha - \beta}} \right\Vert_{L^{l,m}(\R, d\mu_\alpha)}  \omega(r)   r^{\frac{d_\alpha}{p_2} }.\nonumber
\end{align}
Since the above inequality holds for every $x_0 \in \R$ and for every $r  > 0$, we have 
\begin{align}
    ||\mathcal{H}_{\psi,\beta}^\alpha f_2||_{L^{p_2, \omega}(\R, d\mu_\alpha)} \leq C||f||_{L^{p_1, \phi}(\R, d\mu_\alpha)}  \left\Vert \frac{1}{|.|^{d_\alpha - \beta}} \right\Vert_{L^{l,m}(\R, d\mu_\alpha)}. \label{eq: 4.10a}
\end{align}
Combining \eqref{eq: 4.8a} and \eqref{eq: 4.10a}, we complete the proof. \qed

The following theorem is a special case of the above Theorem \ref{4.3}, which establishes the boundedness of the fractional Dunkl-type Hausdorff operator on Dunkl-type Morrey spaces.
\begin{theorem} \label{3.3}
    Let $\psi$ be a measurable function on $\R$ and which satisfy \eqref{eq: 4.4} and  $0 < \beta < d_\alpha$. Let $1 \leq p_1 \leq q_1 < \frac{d_\alpha}{\beta}$, $\frac{1}{p_2} = \frac{1}{p_1} - \frac{1}{l'}$ and $\frac{1}{q_2} = \frac{1}{q_1} - \frac{1}{m'}$ with $1 \leq l < m= \frac{d_\alpha}{d_\alpha -\beta}$. Then $\mathcal{H}_{\psi, \beta}^\alpha$ is bounded from $L^{p_1, q_1}(\R, d\mu_\alpha)$ to $L^{p_2, q_2}(\R, d\mu_\alpha)$ with
        \begin{align}
    ||\mathcal{H}_{\psi,\beta}^\alpha f||_{L^{p_2, q_2}(\R, d\mu_\alpha)} \leq ||f||_{L^{p_1, q_1}(\R, d\mu_\alpha)}  \left\Vert \frac{1}{|.|^{d_\alpha - \beta}} \right\Vert_{L^{l,m}(\R, d\mu_\alpha)}.\nonumber
\end{align} 
\end{theorem}
\textbf{Proof.} The proof of Theorem \ref{3.3} follows from Theorem \ref{4.3} by taking $\phi(r)= r^{-\frac{d_\alpha}{q_1}}$. Then  
\begin{align}
\omega(r)= \phi(r) r^{\frac{d_\alpha}{m'}} = r^{d_\alpha\left(-\frac{1}{q_1} +\frac{1}{m'}\right)}= r^{-\frac{d_\alpha}{q_2}}, \nonumber
\end{align}
where we have used the relation $ \frac{1}{q_2} = \frac{1}{q_1} - \frac{1}{m'} $. Furthermore, using the same relation we observe that the above $\phi$ satisfies the following integral condition:
\begin{align}
   \int_r^\infty \phi(t)t^{\frac{d_\alpha}{m'}-1}\,dt
=\int_r^\infty
t^{-\frac{d_\alpha}{q_1}+\frac{d_\alpha}{m'}-1}\,dt
=C\,\phi(r)r^{\frac{d_\alpha}{m'}}. \nonumber
\end{align}
 Hence, the integral condition in Theorem \ref{4.3} is satisfied. Therefore, all the hypotheses of Theorem \ref{4.3} hold, and the desired boundedness of $\mathcal{H}_{\psi,\beta}^\alpha$ follows immediately from Theorem\ref{4.3}.     \qed 

\section{\bf{Applications}} \label{Sec: 5}
In this section, we provide the applications of some of our main results, which we have proved in earlier sections. Indeed in this section, we will introduce the Dunkl-type Hardy operator, the fractional Dunkl-type Hardy operator and their adjoint operators and prove the boundedness of these operators on Dunkl-type Morrey spaces and Dunkl-type Campanato spaces. Some of these results are new even in the classical case.

\textbf {Application 1.} Let $\psi(t)= \frac{1}{|t|} \chi_{(1, \infty)}(|t|)$. Then from \eqref{eq:3.1a}, we have
\begin{align}
    \mathcal{H}_\psi^\alpha f(x)= \int_{|t| > 1} \frac{1}{|t|^{2\alpha +3}}  f\left(\frac{x}{t} \right)dt. \nonumber
    \end{align}
    By substituting $\frac{x}{t}= u$, the above expression can be expressed as 
    \begin{align}
         \mathcal{H}_\psi^\alpha f(x) = \frac{1}{|x|^{2\alpha + 2}} \int_{|u|<|x|} f(u) |u|^{2 \alpha + 1} du, \hspace{.5cm}x \in \R-\{0\}, \nonumber
    \end{align}
which we refer to as the Dunkl-type Hardy operator and denote it by $\mathcal{H}^\alpha$, i.e.,  
\begin{align}
    \mathcal{H}^\alpha f(x) = \frac{1}{|x|^{2\alpha + 2}} \int_{|u|<|x|} f(u) |u|^{2 \alpha + 1} du, \hspace{.5cm} x \in \R-\{0\}. \label{eq:5.1}
\end{align}
For $\alpha= -1/2$, $\mathcal{H}^\alpha$ reduces to the (classical) Hardy operator $\mathcal{H}$(see \eqref{eq:1.01}). 

In the following corollary, we prove the boundedness of the Dunkl-type Hardy operator on Dunkl-type Morrey spaces and Dunkl-type Campanato spaces.
\begin{Cor} \label{cor:5.1}
If $q > 1$ and $1 \leq p \leq q \leq \infty$, then we have 
\begin{align}
    ||\mathcal{H}^\alpha f||_{L^{p,q}(\R, d\mu_\alpha)} \leq \frac{q}{d_\alpha (q-1)} ||f||_{L^{p,q}(\R, d\mu_\alpha)}, \nonumber
\end{align}
and
\begin{align}
    ||\mathcal{H}^\alpha f||_{\mathcal{L}^{p,q}(\R, d\mu_\alpha)} \leq \frac{q}{d_\alpha (q-1)} ||f||_{\mathcal{L}^{p,q}(\R, d\mu_\alpha)}. \nonumber
\end{align}
\end{Cor} 
\textbf{Proof.} The proof of the first part of Corollary \ref{cor:5.1} follows from Theorem \ref{thm:3.1}, by taking $\psi(t)= \frac{1}{|t|} \chi_{(1, \infty)}(|t|)$. Then 
\begin{align}
    K_\psi &= \int_{|t| > 1} \frac{1}{|t|^{d_\alpha\left(1 - \frac{1}{q} \right) +1}} dt \nonumber
    \\
    &= \frac{q}{d_\alpha(q-1)}, \hspace{0.5cm} \text{if}~ q >1.  \nonumber
\end{align} 
The proof of the second part of Corollary \ref{cor:5.1} follows from Theorem \ref{thm:3.2} by taking same $\psi$ and same $K_\psi$ as above.  This completes the proof of Corollary \ref{cor:5.1}.             \qed

Similarly, when $\psi(t)= |t|^{2\alpha +1}\chi_{(0, 1]}(|t|)$, then again from \eqref{eq:3.1a}, we get 
    \begin{align}
    \mathcal{H}_\psi^\alpha f(x)= \int_{|t| \leq 1} \frac{1}{|t|}  f\left(\frac{x}{t} \right)dt. \nonumber
    \end{align}
By substituting $\frac{x}{t}= u$, the above expression can be expressed as 
\begin{align}
    \mathcal{H}_\psi^\alpha f(x)&= \int_{|u|\geq|x|} \frac{f(u)}{|u|}  du,\nonumber
    \\
    &= \int_{|u|\geq|x|} \frac{f(u)}{|u|^{2\alpha +2}} d\mu_\alpha(u),\nonumber
\end{align}
 which we call the adjoint Dunkl-type Hardy operator and denote it by $(\mathcal{H}^\alpha)^*$, i.e.,
\begin{align}
    (\mathcal{H}^{\alpha})^* f(x) = \int_{|u|\geq|x|} \frac{f(u)}{|u|^{2\alpha +2}} d\mu_\alpha(u). \label{eq:5.2}
\end{align}
For $\alpha= -1/2$, $(\mathcal{H}^\alpha)^*$ reduces to the adjoint Hardy operator $\mathcal{H}^*$ (see \eqref{eq:1.02}).

In the following corollary, we prove the boundedness of the adjoint Dunkl-type Hardy operator on Dunkl-type Morrey spaces and Dunkl-type Campanato spaces.
\begin{Cor} \label{cor:5.2}
    If $1 \leq p \leq q \leq \infty$, then we have 
\begin{align}
    ||(\mathcal{H}^\alpha)^* f||_{L^{p,q}(\R, d\mu_\alpha)} \leq \frac{q}{ d_\alpha} ||f||_{L^{p,q}(\R, d\mu_\alpha)}, \nonumber
\end{align}
and
\begin{align}
    ||(\mathcal{H}^\alpha)^* f||_{\mathcal{L}^{p,q}(\R, d\mu_\alpha)} \leq \frac{q}{d_\alpha } ||f||_{\mathcal{L}^{p,q}(\R, d\mu_\alpha)}. \nonumber
\end{align}
\end{Cor}
\textbf{Proof.} The proof of Corollary \ref{cor:5.2} follows from Theorem \ref{thm:3.1} and Theorem \ref{thm:3.2} respectively by taking $\psi(t) = |t|^{2\alpha + 1}\chi_{(0,1]}(|t|)$. Then 
\begin{align}
    K_\psi &= \int_{|t| \leq 1} \frac{|t|^{2\alpha + 1}}{|t|^{d_\alpha\left(1 - \frac{1}{q} \right)}} dt \nonumber
    \\
    &=  \int_{|t| \leq 1} |t|^{\frac{d_\alpha}{q}-1} dt=\frac{q}{d_\alpha}, \nonumber
\end{align} 
which completes the proof of Corollary \ref{cor:5.2}.\qed

Let $\psi(t)= \frac{|t|^{2\alpha +1}}{\text{max}\{1, |t|^{2\alpha +2}\}}$. Then from \eqref{eq:3.1a}, we have
\begin{align}
    \mathcal{H}_\psi^\alpha f(x) &=\int_\R \frac{|t|^{2\alpha +1}}{\text{max}\{1, |t|^{2\alpha +2}\}} \frac{1}{|t|^{2\alpha +2}} f\left(\frac{x}{t} \right) dt \nonumber
    \\
    &= \int_{|t| \leq 1} \frac{1}{|t|} f\left( \frac{x}{t}\right) dt + \int_{|t| > 1} \frac{1}{|t|^{2\alpha +3}} f\left( \frac{x}{t} \right) dt. \nonumber
    \end{align}
    By substituting $\frac{x}{t} =u$, the above expression can be expressed as 
    \begin{align}
        \mathcal{H}_\psi^\alpha f(x) &= \int_{|x| \leq |u|} \frac{f(u)}{|u|^{2\alpha +2}} d\mu_\alpha(u) + \frac{1}{|x|^{2\alpha + 2}} \int_{|u| < |x|} f(u) d\mu_\alpha(u)  \nonumber
        \\
        & = \int_\R \frac{f(u)}{\text{max}\{ |x|^{2\alpha +2}, |u|^{2\alpha +2}\}} d\mu_\alpha(u), \hspace{0.5cm} x \in \R-\{0\}, \nonumber
    \end{align}
    which we refer to as the Dunkl-type Hardy-Littlewood-P\'olya operator and denote it by $\mathcal{T}^\alpha$, i.e., 
    \begin{align}
       \mathcal{T}^\alpha f(x) := \int_\R \frac{f(u)}{\text{max}\{ |x|^{2\alpha +2}, |u|^{2\alpha +2}\}} d\mu_\alpha(u), \hspace{0.5cm} x \in \R-\{0\}. \label{eq: 5.12} 
    \end{align}
    For $\alpha= -1/2$, $\mathcal{T}^\alpha$ reduces to the Hardy-Littlewood-P\'olya operator $\mathcal{T}$ (see \ref{eq:1.03}). 
    \begin{Cor} \label{Cor: 5.3a}
        If $q > 1 $ and $1 \leq p \leq q < \infty$, then we have 
        \begin{align}
            ||\mathcal{T}^\alpha f ||_{L^{p,q}(\R, d\mu_\alpha)} \leq \frac{q^2}{d_\alpha (q-1)} ||f||_{L^{p,q}(\R, d\mu_\alpha)}, \nonumber 
        \end{align}
        and 
        \begin{align}
            ||\mathcal{T}^\alpha f ||_{\mathcal{L}^{p,q}(\R, d\mu_\alpha)} \leq \frac{q^2}{d_\alpha (q-1)} ||f||_{\mathcal{L}^{p,q}(\R, d\mu_\alpha)}. \nonumber  
        \end{align}
    \end{Cor}
    \textbf{Proof.} The proof of the first part of Corollary \ref{Cor: 5.3a}  follows from Theorem \ref{thm:3.1}, by taking $\psi(t) = \frac{|t|^{2\alpha +1}}{\text{max}\{ 1, |t|^{2\alpha + 2}\}}$. Then 
    \begin{align}
        K_\psi &= \int_\R \frac{|\psi(t)|}{|t|^{d_\alpha \left( 1 - \frac{1}{q}\right)}} dt \nonumber
        \\
        &= \int_\R \frac{|t|^{2\alpha +1}}{\text{max}\{1, |t|^{2\alpha + 2}\}} \frac{1}{|t|^{d_\alpha \left( 1 - \frac{1}{q}\right)}} dt \nonumber
        \\
        &= \int_{|t| \leq 1} \frac{|t|^{2\alpha + 1}}{|t|^{d_\alpha \left( 1 - \frac{1}{q}\right)}} dt + \int_{|t| > 1}  \frac{1}{|t|^{d_\alpha \left( 1 - \frac{1}{q}\right) + 1}} dt \nonumber
        \\
        &= \frac{q}{d_\alpha} + \frac{q}{d_\alpha (q -1)} \hspace{0.5cm} \text{if} q > 1. \nonumber  
    \end{align}
    The proof of the second part of Corollary \ref{Cor: 5.3a} follows from Theorem \ref{thm:3.2} by taking same $\psi$ and same $K_\psi$ as above. This completes the proof of Corollary \ref{Cor: 5.3a}. \qed
    
\textbf {Application 2.} Let $\psi(z)= |z|^{\beta -1} \chi_{(1, \infty)}(|z|)$, $0 \leq \beta < d_\alpha$. Then from \eqref{eq:4.1a} 
\begin{align}
     \mathcal{H}_{\psi, \beta}^\alpha f(x) = \frac{1}{|x|^{2\alpha + 1}} \int_{|z|<|x|} \frac{|xz^{-1}|^{\beta -1}}{|z|^{1-\beta}} f(z) |z|^{2\alpha + 1} dz, \hspace{0.5cm}  x \in \R-\{0\},\nonumber
\end{align}
 which we refer to as the fractional Dunkl-type Hardy operator and denote it by $\mathcal{H}_\beta^\alpha$, i.e.,
\begin{align}
    \mathcal{H}_{ \beta}^\alpha f(x) := \frac{1}{|x|^{2\alpha+2 -\beta}} \int_{|z|<|x|} f(z) |z|^{2\alpha + 1} dz, \hspace{0.5cm} x \in \R-\{0\}. 
\end{align}
For $\beta=0$, $\mathcal{H}_\beta^\alpha$ reduces to the Dunkl-type Hardy operator $\mathcal{H}^\alpha$ (see \ref{eq:5.1}) and for $\alpha= -1/2$, $\mathcal{H}_{\beta}^\alpha$ reduces to the fractional Hardy operator $\mathcal{H}_\beta$(see \eqref{eq:1.401}).

In the following corollaries, we prove the boundedness of the fractional Dunkl-type Hardy operator on Dunkl-type Lebesgue spaces and Dunkl-type Morrey spaces respectively.
\begin{Cor} \label{cor:5.3}
  If $p >1$, $1 \leq q < \infty$, $0 \leq \beta < d_\alpha$ and $s = \frac{d_\alpha}{d_\alpha - \beta}$, then we have
  \begin{align}
      ||\mathcal{H}_\beta^\alpha f ||_{L^q(\R, d\mu_\alpha)} \leq \left( \frac{p}{d_\alpha s (p-1)} \right)^{\frac{1}{s}} ||f||_{L^p(\R, d\mu_\alpha)} \nonumber
  \end{align}
  for all $f \in L^p(\R, d\mu_\alpha)$, where $\frac{1}{q}= \frac{1}{p} - \frac{\beta}{d_\alpha}$.
\end{Cor}
\textbf{Proof.} The proof of Corollary \ref{cor:5.3} follows from Theorem \ref{4.1} by taking $\psi(z)= |z|^{\beta -1} \chi_{(1, \infty)}(|z|)$. Then using the relation $\frac{1}{q} = \frac{1}{p} - \frac{\beta}{d_\alpha}$, we have
\begin{align}
    C_{\psi, s, q, \alpha} &= \left( \int_\R |\psi(z)|^s |z|^{s(\frac{ d_\alpha}{q} - (2\alpha +1)) -1} dz \right)^{1/s} \nonumber
    \\
    &= \left( \int_{|z| > 1} |z|^{s(\beta -1) + s(\frac{d_\alpha}{q} -(2\alpha +1)) -1} dz \right)^{1/s} \nonumber
    \\
    &= \left(\frac{1}{s(1-\beta ) +s(2\alpha +1 -\frac{d_\alpha}{q}) }\right)^{1/s}, \nonumber
    \\
    &\text{when}~ s(\beta -1) + s(\frac{d_\alpha}{q} - (2\alpha +1)) < 0  ~\text{i.e. when}~ p > 1\nonumber
    \\
    &= \left(\frac{1}{s(1-\beta) + s(2\alpha +1 + \beta - \frac{d_\alpha}{p})} \right)^{1/s} \nonumber
    \\
    &= \left(\frac{p}{d_\alpha s(p-1)} \right)^{1/s}, \hspace{0.5cm} \text{if}~ p >1. \nonumber
\end{align}
This completes the proof of Corollary \ref{cor:5.3}. \qed

\begin{Cor} \label{cor:5.4}
     Let $0 < \beta < 1$. Let $1 \leq p_1 \leq q_1 < \frac{d_\alpha}{\beta}$, $\frac{1}{p_2} = \frac{1}{p_1} - \frac{1}{l'}$ and $\frac{1}{q_2} = \frac{1}{q_1} - \frac{1}{m'}$ with $1 \leq l < m= \frac{d_\alpha}{d_\alpha -\beta}$. Then $\mathcal{H}_{\beta}^\alpha$ is bounded from $L^{p_1, q_1}(\R, d\mu_\alpha)$ to $L^{p_2, q_2}(\R, d\mu_\alpha)$ with
        \begin{align}
    ||\mathcal{H}_{\beta}^\alpha f||_{L^{p_2, q_2}(\R, d\mu_\alpha)} \leq ||f||_{L^{p_1, q_1}(\R, d\mu_\alpha)}  \left\Vert \frac{1}{|.|^{d_\alpha - \beta}} \right\Vert_{L^{l,m}(\R, d\mu_\alpha)}.\nonumber
\end{align}
\end{Cor}
\textbf{Proof.} Consider the function
\[
\psi(z)=|z|^{\beta-1}\chi_{(1,\infty)}(|z|).
\]
First, we show that the above function $\psi$ satisfies the assumption \eqref{eq: 4.4} for $0 \leq \beta < 1$. Indeed
\begin{align}
\frac{|\psi(xz^{-1})|}{|x|^{2\alpha +1}|z|^{1-\beta}}
&=
\frac{|xz^{-1}|^{\beta-1}\chi_{(1,\infty)}(|xz^{-1}|)}{|x|^{2\alpha+1}|z|^{1-\beta}}.
\nonumber
\end{align}
When $1 < |xz^{-1}| < \infty$, then $|xz^{-1}|^{\beta - 1} < 1$ and $\frac{1}{|x|^{2\alpha + 1}} < \frac{1}{|z|^{2\alpha + 1}}$. Therefore
\begin{align}
\frac{|\psi(xz^{-1})|}{|x|^{2\alpha +1}|z|^{1-\beta}}
&\leq
\frac{\chi_{(0,|x|)}(|z|)}{|x|^{2\alpha+1}|z|^{1-\beta}}
\nonumber\\
&\leq |z|^{\beta-d_\alpha} \nonumber
\end{align}
and
\begin{align}
\frac{|\psi(xz^{-1})|}{|x|^{2\alpha +1}|z|^{1-\beta}}
&=
\frac{|xz^{-1}|^{\beta-1}\chi_{(1,\infty)}(|xz^{-1}|)}{|x|^{2\alpha +1}|z|^{1-\beta}}
\nonumber\\
&\leq \frac{|xz^{-1}|^{\beta-1}}{|x|^{2\alpha +1}|z|^{1-\beta}} \nonumber
\\
&=
|x|^{\beta-d_\alpha}. \nonumber
\end{align}
Hence
\begin{align}
\frac{|\psi(xz^{-1})|}{|x|^{2\alpha +1}|z|^{1-\beta}}
\leq
\min\left\{|z|^{\beta-d_\alpha},\,|x|^{\beta-d_\alpha}\right\}.
\nonumber
\end{align}
Therefore, the function $\psi(z)=|z|^{\beta -1}\chi_{(1,\infty)}(|z|)$ satisfies the assumption \eqref{eq: 4.4} for $0 \leq \beta < 1$. 
Finally, the proof of Corollary \ref{cor:5.4} follows from Theorem \ref{3.3} by taking $\psi(z)= |z|^{\beta -1} \chi_{(1, \infty)}(|z|)$ , $0 \leq \beta <1$.  \qed

Similarly, when $\psi(z)= |z|^{2\alpha +1} \chi_{(0, 1]}(|z|)$, then again from \eqref{eq:4.1a}, we get
\begin{align}
\mathcal{H}_{\psi, \beta}^\alpha f(x) &= \frac{1}{|x|^{2 \alpha + 1}} \int_{|z| \geq |x|} \frac{|xz^{-1}|^{2\alpha +1}}{|z|^{1-\beta}} f(z)|z|^{2 \alpha + 1} dz,  \hspace{0.5cm} x \in \R - \{0\},\nonumber
\\
&= \int_{|z| \geq |x|} \frac{f(z)}{|z|^{2\alpha + 2-\beta}} |z|^{2 \alpha + 1} dz,\nonumber
\end{align}
which we refer to as the adjoint fractional Dunkl-type Hardy operator and denote it by $(\mathcal{H}_\beta^\alpha)^*$, i.e., 
\begin{align}
    (\mathcal{H}_\beta^\alpha)^* f(x) :=\int_{|z| \geq |x|} \frac{f(z)}{|z|^{2\alpha + 2-\beta}} |z|^{2 \alpha + 1} dz. \nonumber
\end{align}
For $\beta=0$, $(\mathcal{H}_\beta^\alpha)^*$ reduces to the adjoint Dunkl-type Hardy operator $(\mathcal{H}^\alpha)^*$ (see \ref{eq:5.2}) and for $\alpha= -1/2$, $(\mathcal{H}_\beta^\alpha)^*$ reduces to the adjoint fractional Hardy operator $\mathcal{H}_\beta^*$(see \eqref{eq:1.402}).

In the following corollaries, we prove the boundedness of the adjoint fractional Dunkl-type Hardy operator on Dunkl-type Lebesgue spaces and Dunkl-type Morrey spaces respectively.
\begin{Cor} \label{cor:5.5}
    If $1 \leq p <  \frac{d_\alpha}{\beta }$, $1\leq  q < \infty$, $0 \leq \beta < d_\alpha$ and $s= \frac{d_\alpha}{d_\alpha - \beta}$, then we have
    \begin{align}
      ||(\mathcal{H}_\beta^\alpha)^* f ||_{L^q(\R, d\mu_\alpha)} \leq \left(\frac{p}{s(d_\alpha - p\beta)}\right)^{\frac{1}{s}} ||f||_{L^p(\R, d\mu_\alpha)}, \nonumber
  \end{align}
  for all $f \in L^p(\R, d\mu_\alpha)$, where $\frac{1}{q}= \frac{1}{p} - \frac{\beta}{d_\alpha}$.
\end{Cor}
\textbf{Proof.} The proof of Corollary \ref{cor:5.5} follows from Theorem \ref{4.1} by taking $\psi(z)= |z|^{2\alpha + 1}\chi_{(0, 1]}(|z|)$. Then  using the relation $\frac{1}{q}= \frac{1}{p} - \frac{\beta}{d_\alpha}$, we have
\begin{align}
    C_{\psi, s, q, \alpha} &= \left( \int_\R |\psi(z)|^s |z|^{s(\frac{ d_\alpha}{q} - (2\alpha +1)) -1} dz \right)^{1/s} \nonumber
    \\
    &= \left( \int_{|z| \leq 1} |z|^{s(2\alpha+1)} |z|^{s( \frac{d_\alpha}{q} -(2 \alpha +1) ) -1} dz \right)^{1/s} \nonumber
    \\
    &= \left( \int_{|z| \leq 1}  |z|^{ \frac{s d_\alpha}{q} -1} dz \right)^{1/s} \nonumber
    \\
    &= \left(\frac{q}{sd_\alpha}\right)^{1/s}, \hspace{0.5cm} \text{when}~ \frac{s d_\alpha}{q} > 0 ~\text{i.e. when}~ p < \frac{d_\alpha}{\beta } \nonumber
    \\
    &= \left(\frac{p}{s(d_\alpha - p\beta)}\right)^{1/s}, \hspace{0.5cm} \text{if}~ p < \frac{d_\alpha}{\beta}. \nonumber
\end{align}
This completes the proof of Corrolary \ref{cor:5.5}. \qed

\begin{Cor} \label{cor:5.6}
     Let $0 < \beta < 1$. Let $1 \leq p_1 \leq q_1 < \frac{d_\alpha}{\beta}$, $\frac{1}{p_2} = \frac{1}{p_1} - \frac{1}{l'}$ and $\frac{1}{q_2} = \frac{1}{q_1} - \frac{1}{m'}$ with $1 \leq l < m= \frac{d_\alpha}{d_\alpha -\beta}$. Then $(\mathcal{H}_{\beta}^\alpha)^*$ is bounded from $L^{p_1, q_1}(\R, d\mu_\alpha)$ to $L^{p_2, q_2}(\R, d\mu_\alpha)$ with
        \begin{align}
    ||(\mathcal{H}_{\beta}^\alpha)^* f||_{L^{p_2, q_2}(\R, d\mu_\alpha)} \leq ||f||_{L^{p_1, q_1}(\R, d\mu_\alpha)}  \left\Vert \frac{1}{|.|^{d_\alpha - \beta}} \right\Vert_{L^{l,m}(\R, d\mu_\alpha)}.\nonumber
\end{align}
\end{Cor}
\textbf{Proof.} Consider the function
\[
\psi(z)= |z|^{2\alpha + 1}\chi_{(0,1]}(|z|).
\]
First, we show that the above function $\psi$ satisfies the assumption \eqref{eq: 4.4} for $0 \leq \beta  <1$. Indeed
\begin{align}
\frac{|\psi(xz^{-1})|}{|x|^{2\alpha +1}|z|^{1-\beta}}
&=
\frac{|x z^{-1}|^{2\alpha +1}\chi_{(0,1]}(|xz^{-1}|)}{|x|^{2\alpha+1}|z|^{1-\beta}}.
\nonumber
\end{align}
When $0 < |xz^{-1}| \leq 1$, then $|xz^{-1}|^{2\alpha + 1} \leq 1$ and $\frac{1}{|x|^{1-\beta}} \geq \frac{1}{|z|^{1-\beta}}$. Therefore
\begin{align}
\frac{|\psi(xz^{-1})|}{|x|^{2\alpha +1}|z|^{1-\beta}}
&\leq
\frac{\chi_{[|x|,\infty)}(|z|)}{|x|^{2\alpha+1}|z|^{1-\beta}}
\nonumber
\\
&\leq
\frac{\chi_{[|x|,\infty)}(|z|)}{|x|^{2\alpha+1}|x|^{1-\beta}}
\nonumber
\\
&\leq |x|^{\beta-d_\alpha} \nonumber
\end{align}
and
\begin{align}
\frac{|\psi(xz^{-1})|}{|x|^{2\alpha +1}|z|^{1-\beta}}
&=
\frac{|xz^{-1}|^{2\alpha +1}\chi_{(0,1]}(|xz^{-1}|)}{|x|^{2\alpha +1}|z|^{1-\beta}}
\nonumber\\
&\leq \frac{|xz^{-1}|^{2\alpha +1}}{|x|^{2\alpha +1}|z|^{1-\beta}} \nonumber
\\
&=
|z|^{\beta-d_\alpha}. \nonumber
\end{align}
Hence
\begin{align}
\frac{|\psi(xz^{-1})|}{|x|^{2\alpha +1}|z|^{1-\beta}}
\leq
\min\left\{|z|^{\beta-d_\alpha},\,|x|^{\beta-d_\alpha}\right\}.
\nonumber
\end{align}
Therefore, the function $\psi(z)=|z|^{2\alpha +1}\chi_{(0,1]}(|z|)$ satisfies the assumption \eqref{eq: 4.4} for $0 \leq \beta < 1$. Finally, the proof of Corollary \ref{cor:5.6} follows from Theorem \ref{3.3} by taking $\psi(z)= |z|^{2\alpha +1}\chi_{(0, 1]}(|z|)$.  \qed

Let $\psi(z)= \frac{|z|^{2\alpha + 1}}{\text{max} \{ 1, |z|^{d_\alpha -\beta}\}}$, $0 \leq \beta < d_\alpha$. Then from \eqref{eq:4.1a} 
\begin{align}
     \mathcal{H}_{\psi, \beta}^\alpha f(x) &= \frac{1}{|x|^{2\alpha + 1}} \int_{\R} \frac{|xz^{-1}|^{2\alpha + 1}}{\text{max} \{1, |xz^{-1}|^{d_\alpha - \beta}\}} \frac{1}{|z|^{1-\beta}} f(z) |z|^{2\alpha + 1} dz \nonumber
     \\
     &= \int_{|xz^{-1}| \leq 1} \frac{|xz^{-1}|^{2\alpha + 1}}{|z|^{1 -\beta}} f(z)d\mu_\alpha(z) \nonumber
     \\
     & \hspace{0.5cm} + \frac{1}{|x|^{2\alpha +1}} \int_{|xz^{-1}| >1} \frac{|xz^{-1}|^{2\alpha + 1}}{|xz^{-1}|^{d_\alpha -\beta} |z|^{1-\beta}} f(z) |z|^{2\alpha + 1} dz\nonumber
     \\
     &= \int_{|x| \leq |z|} \frac{f(z)}{|z|^{2\alpha + 2 -\beta}} d\mu_\alpha(z) + \frac{1}{|x|^{2\alpha +2 -\beta}} \int_{|z| < |x|} f(z) |z|^{2\alpha + 1} dz\nonumber
     \\
     &= \int_\R \frac{f(z)}{\text{max} \{ |x|^{d_\alpha -\beta}, |z|^{d_\alpha - \beta} \}} d\mu_\alpha(z), \hspace{0.5cm} x \in \R-\{0\}, \nonumber
\end{align}
 which we refer to as the fractional Dunkl-type Hardy-Littlewood-P\'olya operator and denote it by $\mathcal{T}_\beta^\alpha$, i.e.,
\begin{align}
    \mathcal{T}_{ \beta}^\alpha f(x) :=  \int_\R \frac{f(z)}{\text{max} \{ |x|^{d_\alpha -\beta}, |z|^{d_\alpha - \beta} \}} d\mu_\alpha(z), \hspace{0.5cm} x \in \R-\{0\}. \nonumber
\end{align}
For $\beta=0$, $\mathcal{T}_\beta^\alpha$ reduces to the Dunkl-type Hardy-Littlewood-P\'olya operator $\mathcal{T}^\alpha$ (see \ref{eq: 5.12}) and for $\alpha= -1/2$, $\mathcal{T}_{\beta}^\alpha$ reduces to the fractional Hardy-Littlewood-P\'olya operator $\mathcal{T}_\beta$(see \eqref{eq: 1.07}).

In the following corollaries, we prove the boundedness of the fractional Dunkl-type Hardy-Littlewood-P\'olya operator on Dunkl-type Lebesgue spaces and Dunkl-type Morrey spaces respectively.
\begin{Cor} \label{cor:5.8}
  If $1 < p < \frac{d_\alpha}{\beta}$, $1 \leq q <\infty$, $0 \leq \beta < d_\alpha$ and $s = \frac{d_\alpha}{d_\alpha - \beta}$, then we have
  \begin{align}
      ||\mathcal{T}_\beta^\alpha f ||_{L^q(\R, d\mu_\alpha)} \leq \left(\frac{p}{s(d_\alpha - p\beta)} + \frac{p}{d_\alpha s (p-1)} \right)^{\frac{1}{s}} ||f||_{L^p(\R, d\mu_\alpha)} \nonumber
  \end{align}
  for all $f \in L^p(\R, d\mu_\alpha)$, where $\frac{1}{q}= \frac{1}{p} - \frac{\beta}{d_\alpha}$.
\end{Cor}
\textbf{Proof.} The proof of Corollary \ref{cor:5.8} follows from Theorem \ref{4.1} by taking $\psi(z)= \frac{|z|^{2\alpha + 1}}{\text{max} \{1, |z|^{d_\alpha -\beta}\}}$. Then using the relation $\frac{1}{q} = \frac{1}{p} - \frac{\beta}{d_\alpha}$, we have
\begin{align}
    C_{\psi, s, q, \alpha} &= \left( \int_\R |\psi(z)|^s |z|^{s(\frac{ d_\alpha}{q} - (2\alpha +1)) -1} dz \right)^{1/s} \nonumber
    \\
    &= \left( \int_\R \frac{|z|^{(2\alpha + 1)s}}{(\text{max}\{ 1, |z|^{d_\alpha -\beta}\})^s} |z|^{s(\frac{d_\alpha}{q} -(2\alpha +1))-1} dz \right)^{1/s} \nonumber
    \\
    &= \left( \int_{|z| \leq 1} |z|^{s(2\alpha+1)} |z|^{s( \frac{d_\alpha}{q} -(2 \alpha +1) ) -1} dz   + \int_{|z| > 1} |z|^{s(\beta -1) + s(\frac{d_\alpha}{q} -(2\alpha +1)) -1} dz \right)^{1/s}. \nonumber
    \end{align}
    Using the same calculation which we have done in Corollary \ref{cor:5.3} and Corollary \ref{cor:5.5}, we have
    \begin{align}
     C_{\psi, s, q, \alpha} &= \left(\frac{p}{s(d_\alpha - p\beta)} +\frac{p}{d_\alpha s(p-1)} \right)^{1/s}, \hspace{0.5cm} \text{if}~ 1 < p < \frac{d_\alpha}{\beta}. \nonumber
\end{align}
This completes the proof of Corollary \ref{cor:5.8}. \qed

\begin{Cor} \label{cor:5.9}
     Let $0 < \beta < 1$. Let $1 \leq p_1 \leq q_1 < \frac{d_\alpha}{\beta}$, $\frac{1}{p_2} = \frac{1}{p_1} - \frac{1}{l'}$ and $\frac{1}{q_2} = \frac{1}{q_1} - \frac{1}{m'}$ with $1 \leq l < m= \frac{d_\alpha}{d_\alpha -\beta}$. Then $\mathcal{T}_{\beta}^\alpha$ is bounded from $L^{p_1, q_1}(\R, d\mu_\alpha)$ to $L^{p_2, q_2}(\R, d\mu_\alpha)$ with
        \begin{align}
    ||\mathcal{T}_{\beta}^\alpha f||_{L^{p_2, q_2}(\R, d\mu_\alpha)} \leq ||f||_{L^{p_1, q_1}(\R, d\mu_\alpha)}  \left\Vert \frac{1}{|.|^{d_\alpha - \beta}} \right\Vert_{L^{l,m}(\R, d\mu_\alpha)}.\nonumber
\end{align}
\end{Cor}
\textbf{Proof.} Consider the function
\[
\psi(z)= \frac{|z|^{2\alpha + 1}}{\text{max} \{ 1, |z|^{d_\alpha -\beta}\}}.
\]
First, we show that the above function $\psi$ satisfies the assumption \eqref{eq: 4.4} for $0 \leq \beta < 1$. Indeed
\begin{align}
\frac{|\psi(xz^{-1})|}{|x|^{2\alpha +1}|z|^{1-\beta}}
&=
\frac{|xz^{-1}|^{2\alpha + 1}}{\text{max} \{ 1, |xz^{-1}|^{d_\alpha -\beta}\}} \frac{1}{|x|^{2\alpha+1}|z|^{1-\beta}}
\nonumber
\\
&= \begin{cases}
        \frac{|xz^{-1}|^{2\alpha + 1}}{|x|^{2\alpha +1} |z|^{1-\beta}}, & \text{if} ~ 0 < |xz^{-1}| \leq 1,\\
        \frac{|x z^{-1}|^{2\alpha+1}}{|xz^{-1}|^{d_\alpha -\beta} |x|^{2\alpha +1}|z|^{1-\beta}}, & \text{it}~ 1 < |xz^{-1}| < \infty
    \end{cases} \nonumber
    \\
    &=\begin{cases}
        |z|^{\beta- d_\alpha},  & \text{if} ~ 0 < |xz^{-1}| \leq 1,\\
        \frac{|x z^{-1}|^{\beta-1}}{|x|^{2\alpha +1}|z|^{1-\beta}}, & \text{it}~ 1 < |xz^{-1}| < \infty.
    \end{cases} \label{eq:5.5}
\end{align}
Now when $1 < |xz^{-1}| < \infty$, then $|xz^{-1}|^{\beta - 1} < 1$ and $\frac{1}{|x|^{2\alpha + 1}} < \frac{1}{|z|^{2\alpha + 1}}$. Therefore from \eqref{eq:5.5},
\begin{align}
\frac{|\psi(xz^{-1})|}{|x|^{2\alpha +1}|z|^{1-\beta}}
&\leq
\begin{cases}
        |z|^{\beta- d_\alpha}, & \text{if} ~  |x| \leq |z|,\\
        \frac{1}{ |z|^{2\alpha +1}|z|^{1-\beta}}, & \text{if}~ |x| > |z| < \infty
    \end{cases} \nonumber\\
&\leq |z|^{\beta-d_\alpha}. \nonumber
\end{align}
Also
\begin{align}
\frac{|\psi(xz^{-1})|}{|x|^{2\alpha +1}|z|^{1-\beta}}
&=
\frac{|xz^{-1}|^{2\alpha + 1}}{\text{max} \{ 1, |xz^{-1}|^{d_\alpha -\beta}\}} \frac{1}{|x|^{2\alpha+1}|z|^{1-\beta}}.
\nonumber
\\
&= \begin{cases}
        \frac{|xz^{-1}|^{2\alpha + 1}}{|x|^{2\alpha +1} |z|^{1-\beta}}, & \text{if} ~ 0 < |xz^{-1}| \leq 1,\\
        \frac{|x z^{-1}|^{2\alpha+1}}{|xz^{-1}|^{d_\alpha -\beta} |x|^{2\alpha +1}|z|^{1-\beta}}, & \text{it}~ 1 < |xz^{-1}| < \infty
    \end{cases} \nonumber
    \\
    &=\begin{cases}
        \frac{|xz^{-1}|^{2\alpha +1}}{|x|^{2\alpha +1} |z|^{1-\beta}}, & \text{if} ~ 0 < |xz^{-1}| \leq 1,\\
        |x|^{\beta - d_\alpha},  & \text{if}~ 1 < |xz^{-1}| < \infty.
    \end{cases} \label{eq:5.6}
\end{align}
Again when $0 < |xz^{-1}| \leq 1$, then $|xz^{-1}|^{2\alpha + 1} \leq 1$ and $\frac{1}{|x|^{1-\beta}} \geq \frac{1}{|z|^{1-\beta}}$. Therefore from \eqref{eq:5.6}
\begin{align}
\frac{|\psi(xz^{-1})|}{|x|^{2\alpha +1}|z|^{1-\beta}}
&\leq
\begin{cases}
        \frac{1}{|x|^{2\alpha +1} |x|^{1-\beta}}, & \text{if} ~  |x| \leq |z|,\\
        |x|^{\beta - d_\alpha}, & \text{if}~ |x| > |z| < \infty
    \end{cases} \nonumber\\
&\leq |x|^{\beta-d_\alpha}. \nonumber
\end{align}
Hence
\begin{align}
\frac{|\psi(xz^{-1})|}{|x|^{2\alpha +1}|z|^{1-\beta}}
\leq
\min\left\{|z|^{\beta-d_\alpha},\,|x|^{\beta-d_\alpha}\right\}.
\nonumber
\end{align}
Therefore, the function $\psi(z)= \frac{|z|^{2\alpha + 1}}{\text{max} \{ 1, |z|^{d_\alpha -\beta}\}}$ satisfies the assumption \eqref{eq: 4.4} for $0 \leq \beta < 1$. 
Finally, the proof of Corollary \ref{cor:5.9} follows from Theorem \ref{3.3} by taking $\psi(z)= \frac{|z|^{2\alpha + 1}}{\text{max} \{ 1, |z|^{d_\alpha -\beta}\}}$ , $0 \leq \beta <1$.  \qed


 \section{\bf{Norm estimate of Hausdorff operators associated with Dunkl operator on the real line}} \label{sec:6}

 In this section we find the norm of Dunkl-type Hausdorff operator and fractional Dunkl-type Hausdorff operator on generalized Dunkl-type Morrey spaces.

\begin{theorem} \label{thm:6.2}
Let $1 \leq p < q < \infty$ and let $\psi$ be non-negative function defined on $\R$. Then the Dunkl-type Hausdorff operator $\mathcal{H}_\psi^\alpha$ is bounded on the Dunkl-type Morrey space $L^{p,q}(\R, d\mu_\alpha)$ if and only if 
    \begin{align}
        K_\psi = \int_\R \frac{\psi(t)}{|t|^{d_\alpha\left(1 - \frac{1}{q}\right)}} dt < \infty. \nonumber
    \end{align}
     Furthermore   
 \begin{align}
     ||\mathcal{H}_\psi^\alpha||_{L^{p,q}(\R, d\mu_\alpha) \to L^{p,q}(\R, d\mu_\alpha)} = K_\psi. \nonumber
 \end{align}
\end{theorem}
\textbf{Proof} By Theorem \ref{thm:3.1}, we already have 
\begin{align}
    ||\mathcal{H}_\psi^\alpha f||_{L^{p,q}(\R, d\mu_\alpha)} \leq K_{\psi} ||f||_{L^{p,q}(\R, d\mu_\alpha)},~ \forall f \in L^{p,q}(\R, d\mu_\alpha). \nonumber
\end{align}
In order to prove the reverse statement, we need to find a non-zero function $g \in L^{p, q}(\R, d\mu_\alpha)$ such that
\begin{align}
    \frac{||{\mathcal{H}_\psi^\alpha} g||_{L^{p, q}(\R, d\mu_\alpha)}}{||g||_{L^{p, q}(\R, d\mu_\alpha)}}= K_{ \psi}. \nonumber
\end{align}
We take $g(y)= |y|^{-\frac{d_\alpha}{q}}$. First, we show that $g \in L^{p,q}(\R, d\mu_\alpha)$ i.e., from the definition of the Dunkl-type Morrey norm, we show
\begin{align}
    || g ||_{L^{p,q}(\R, d\mu_\alpha)} = \sup_{\substack{{r>0}\\{x \in \R}}} r^{d_\alpha\left( \frac{1}{q} - \frac{1}{p} \right)} \left(\int_{B(0,r)} \tau_x^\alpha |.|^{-\frac{d_\alpha p}{q}}(y) d\mu_\alpha(y)\right)^{\frac{1}{p}} < \infty. \label{eq:6.1a}
\end{align}
Now, we consider the integral term in \eqref{eq:6.1a}
\begin{align}
    \int_{B(0,r)} \tau_x^\alpha |.|^{-\frac{d_\alpha p}{q}}(y) d\mu_\alpha(y) &= \int_{\R}   |y|^{-\frac{d_\alpha p}{q}} \tau_{-x}^\alpha \chi_{B(0,r)}(y) d\mu_\alpha(y)\nonumber
    \\
    &=  \int_{B(-x,r)}  |y|^{-\frac{d_\alpha p}{q}} \tau_{-x}^\alpha \chi_{B(0,r)}(y) d\mu_\alpha(y), \label{eq:6.2a}
\end{align}
where we have used the support of $\tau_{-x}^\alpha \chi_{B(0,r)}$ from Lemma \ref{lem: 2.2}. We consider two cases:
\begin{itemize}
    \item [\bf Case 1.] If $|x| > 2r$, then 
    \begin{align}
       B(-x, r) & = \{y \in \mathbb{R} : |y| \in \,]max\{0, |x|-r\}, |x| + r[\,\}\nonumber
       \\
       &= \{ y \in \R : |y| \in ]  |x| - r, |x| + r[\}. \nonumber 
    \end{align}
    This implies for $y \in B(-x,r)$, $|y| >  |x|-r > r$. Then for $1 \leq p \leq q < \infty$, from \eqref{eq:6.2a}, we have 
    \begin{align}
           \int_{B(0,r)} \tau_x^\alpha |.|^{-\frac{d_\alpha p}{q}}(y) d\mu_\alpha(y) &= \int_{B(-x,r)}  |y|^{-\frac{d_\alpha p}{q}} \tau_{-x}^\alpha \chi_{B(0,r)}(y) d\mu_\alpha(y) \nonumber
           \\
           &\leq \int_{B(-x,r)}   r^{-\frac{d_\alpha p}{q}} \tau_{-x}^\alpha \chi_{B(0,r)}(y) d\mu_\alpha(y) \nonumber
           \\
           &\leq r^{-\frac{d_\alpha p}{q}} \int_{\R}    \tau_{-x}^\alpha \chi_{B(0,r)}(y) d\mu_\alpha(y)= r^{d_\alpha-\frac{d_\alpha p}{q}} . \label{eq:6.3a}
    \end{align}
    Therefore from \eqref{eq:6.3a} and \eqref{eq:6.1a}, we obtain
    \begin{align}
       r^{d_\alpha\left( \frac{1}{q} - \frac{1}{p} \right)} \left(\int_{B(0,r)} \tau_x^\alpha |.|^{-\frac{d_\alpha p}{q}}(y) d\mu_\alpha(y)\right)^{\frac{1}{p}} \leq 1. \nonumber
    \end{align}
    \item[\bf Case 2.] If $|x| \leq 2r$, then 
     \begin{align}
       B(-x, r) & = \{y \in \mathbb{R} : |y| \in \,]max\{0, |x|-r\}, |x| + r[\,\}\nonumber
       \\
       &\subseteq \{ y \in \R : |y| \in ]  0, |x| + r[\} \subseteq B(0, 3r). \nonumber
    \end{align}
    Thus from \eqref{eq:6.2a}, we get
    \begin{align}
         \int_{B(0,r)} \tau_x^\alpha |.|^{-\frac{d_\alpha p}{q}}(y) d\mu_\alpha(y) &=  \int_{B(-x,r)}  |y|^{-\frac{d_\alpha p}{q}} \tau_{-x}^\alpha \chi_{B(0,r)}(y) d\mu_\alpha(y) \nonumber
         \\
        &\leq \int_{B(0,3r)}  |y|^{-\frac{d_\alpha p}{q}} \tau_{-x}^\alpha \chi_{B(0,r)}(y) d\mu_\alpha(y) \nonumber
        \\
        &\leq \int_{B(0,3r)}  |y|^{-\frac{d_\alpha p}{q}}  d\mu_\alpha(y) \nonumber
        \\
         &\leq \int_{B(0,3r)}  |y|^{-\frac{d_\alpha p}{q}}  |y|^{2\alpha + 1} dy \nonumber
         \\
         &= 2 \frac{(3r)^{d_\alpha \left( 1 - \frac{p}{q} \right)}}{d_\alpha \left( 1 -\frac{p}{q} \right)}, \hspace{0.5cm} p < q, \label{eq:6.4a}
         \end{align}
    where we have used the bound of Dunkl translation from Lemma \ref{lem: 2.2}.  By putting  \eqref{eq:6.4a} in \eqref{eq:6.1a}, we obtain
     \begin{align}
       r^{d_\alpha\left( \frac{1}{q} - \frac{1}{p} \right)} \left(\int_{B(0,r)} \tau_x^\alpha |.|^{-\frac{d_\alpha p}{q}}(y) d\mu_\alpha(y)\right)^{\frac{1}{p}} \leq C. \nonumber
    \end{align}
\end{itemize}
Therefore combining both the above cases we get $||g||_{L^{p,q}(\R, d\mu_\alpha)}$ is finite. Hence $g \in L^{p, q}(\R, d\mu_\alpha)$. Also observe that
\begin{align}
     \mathcal{H}_{\psi}^\alpha g(x) = \int_\R \frac{\psi(t)}{|t|^{2\alpha +2}} \left|\frac{x}{t} \right|^{-\frac{d_\alpha }{q}}  dt = |x|^{-\frac{d_\alpha }{q}} \int_\R \frac{\psi(t)}{|t|^{d_\alpha \left( 1 -\frac{1}{q} \right)}}   dt = K_\psi g(x),\nonumber
\end{align}
which implies
\begin{align}
    ||\mathcal{H}_\psi^\alpha g||_{L^{p,q}(\R, d\mu_\alpha)}= K_\psi ||g||_{L^{p,q}(\R, d\mu_\alpha)}. \nonumber
\end{align}
Hence $K_\psi < \infty$.
This completes the proof.  \qed

\begin{theorem} 
Let $1 \leq p < q < \infty$ and let $\psi$ be non-negative function defined on $\R$. Then the Dunkl-type Hausdorff operator $\mathcal{H}_\psi^\alpha$ is bounded on the Dunkl-type Campanato space $\mathcal{L}^{p,q}(\R, d\mu_\alpha)$ if and only if 
    \begin{align}
        K_\psi = \int_\R \frac{\psi(t)}{|t|^{d_\alpha\left(1 - \frac{1}{q}\right)}} dt < \infty. \nonumber
    \end{align}
     Furthermore   
 \begin{align}
     ||\mathcal{H}_\psi^\alpha||_{\mathcal{L}^{p,q}(\R, d\mu_\alpha) \to \mathcal{L}^{p,q}(\R, d\mu_\alpha)} = K_\psi. \nonumber
 \end{align}
\end{theorem}
\textbf{Proof.} By Theorem \ref{thm:3.2}, we already have 
\begin{align}
    ||\mathcal{H}_\psi^\alpha f||_{\mathcal{L}^{p,q}(\R, d\mu_\alpha)} \leq K_{\psi} ||f||_{\mathcal{L}^{p,q}(\R, d\mu_\alpha)}, ~\forall f \in \mathcal{L}^{p,q}(\R, d\mu_\alpha). \nonumber
\end{align}
In order to prove the reverse statement, we need to find a non-zero function $g \in \mathcal{L}^{p, q}(\R, d\mu_\alpha)$ such that
\begin{align}
    \frac{||{\mathcal{H}_\psi^\alpha} g||_{\mathcal{L}^{p, q}(\R, d\mu_\alpha)}}{||g||_{\mathcal{L}^{p, q}(\R, d\mu_\alpha)}}= K_{ \psi}. \nonumber
\end{align}
Let $g(y)= |y|^{-\frac{d_\alpha}{q}}$ for $q >1$. Then from \eqref{eq:2.6ab}, 
\begin{align}
    g^\alpha_{B(0,r)}(x) &= \frac{1}{\mu_\alpha(B(0,r))} \int_{B(0,r)} \tau_x^\alpha |.|^{-\frac{d_\alpha}{q}}(y) d\mu_\alpha(y) \nonumber
    \\
    &=\frac{1}{\mu_\alpha(B(0,r))}\int_{\R}   |y|^{-\frac{d_\alpha }{q}} \tau_{-x}^\alpha \chi_{B(0,r)}(y) d\mu_\alpha(y)\nonumber
    \\
    &=  \frac{1}{\mu_\alpha(B(0,r))}\int_{B(-x,r)}  |y|^{-\frac{d_\alpha }{q}} \tau_{-x}^\alpha \chi_{B(0,r)}(y) d\mu_\alpha(y). \label{eq:6.9}
\end{align}
 We consider two cases:
\begin{itemize}
    \item [\bf Case 1.] If $|x| > 2r$, then 
    \begin{align}
       B(-x, r) & = \{y \in \mathbb{R} : |y| \in \,]max\{0, |x|-r\}, |x| + r[\,\}\nonumber
       \\
       &= \{ y \in \R : |y| \in ]  |x| - r, |x| + r[\}. \nonumber 
    \end{align}
    Therefore for $y \in B(-x,r)$, $|y| >  |x|-r > r$. Thus from \eqref{eq:6.9}, we have 
    \begin{align}
            g^\alpha_{B(0,r)}(x) &= \frac{1}{\mu_\alpha(B(0,r))} \int_{B(-x,r)}  |y|^{-\frac{d_\alpha }{q}} \tau_{-x}^\alpha \chi_{B(0,r)}(y) d\mu_\alpha(y) \nonumber
           \\
           &\leq r^{-d_\alpha}\int_{B(-x,r)}   r^{-\frac{d_\alpha }{q}} \tau_{-x}^\alpha \chi_{B(0,r)}(y) d\mu_\alpha(y) \nonumber
           \\
           &\leq r^{-\frac{d_\alpha }{q}-d_\alpha} \int_{\R}    \tau_{-x}^\alpha \chi_{B(0,r)}(y) d\mu_\alpha(y)= r^{-\frac{d_\alpha }{q}}  < \infty.  \label{eq:6.7a}
    \end{align}
    \item[\bf Case 2.] If $|x| \leq 2r$, then 
     \begin{align}
       B(-x, r) & = \{y \in \mathbb{R} : |y| \in \,]max\{0, |x|-r\}, |x| + r[\,\}\nonumber
       \\
       &\subseteq \{ y \in \R : |y| \in ]  0, |x| + r[\} \subseteq B(0, 3r). \nonumber
    \end{align}
    Then again from \eqref{eq:6.9}, we get
    \begin{align}
         g^\alpha_{B(0,r)}(x) &= \frac{1}{\mu_\alpha(B(0,r))}  \int_{B(-x,r)}  |y|^{-\frac{d_\alpha }{q}} \tau_{-x}^\alpha \chi_{B(0,r)}(y) d\mu_\alpha(y) \nonumber
         \\
        &\leq r^{-d_\alpha}\int_{B(0,3r)}  |y|^{-\frac{d_\alpha }{q}} \tau_{-x}^\alpha \chi_{B(0,r)}(y) d\mu_\alpha(y) \nonumber
        \\
        &\leq r^{-d_\alpha} \int_{B(0,3r)}  |y|^{-\frac{d_\alpha }{q}}  d\mu_\alpha(y) \nonumber
         \\
         &= 2 r^{-d_\alpha}\frac{(3r)^{d_\alpha \left( 1 - \frac{1}{q} \right)}}{d_\alpha \left( 1 -\frac{1}{q} \right)} \nonumber
         \\
         &=2 \frac{3^{d_\alpha \left( 1 - \frac{1}{q} \right)}}{d_\alpha \left( 1 -\frac{1}{q} \right)} r^{ - \frac{d_\alpha}{q}} < \infty, \hspace{0.5cm} 1 < q. \label{eq:6.8a}
    \end{align}
\end{itemize}
Therefore combining both the above cases we get $g^\alpha_{B(0,r)}(x)$ is finite.
Next, we show that $g \in \mathcal{L}^{p,q}(\R, d\mu_\alpha)$ i.e., from the definition of the Dunkl-type Campanato norm (see \eqref{eq: 7.1}), we need to show
\begin{align}
    || g ||_{\mathcal{L}^{p,q}(\R, d\mu_\alpha)} = \sup_{\substack{{r>0}\\{x \in \R}}} r^{d_\alpha\left( \frac{1}{q} - \frac{1}{p} \right)} \left(\int_{B(0,r)} |\tau_x^\alpha g(y) -g^\alpha_{B(0,r)}(x) |^p d\mu_\alpha(y)\right)^{\frac{1}{p}} < \infty. \label{eq:6.12}
\end{align}
Now, we consider the integral term in \eqref{eq:6.12}
\begin{align}
  & \int_{B(0,r)} |\tau_x^\alpha g(y) -g^\alpha_{B(0,r)}(x) |^p d\mu_\alpha(y) \nonumber
  \\
  &\leq  2^{p-1}\int_{B(0,r)} (|\tau_x^\alpha g(y)|^p + |g^\alpha_{B(0,r)}(x)|^p) d\mu_\alpha(y)\nonumber
      \\
    &= 2^{p-1} \left (  |g^\alpha_{B(0,r)}(x)|^p r^{d_\alpha}+ \int_{B(0,r)} |\tau_x^\alpha |.|^{-\frac{ d_\alpha}{q}}|^p (y)  d\mu_\alpha(y)  \right). \nonumber
    \end{align}
      Using the relation of the Dunkl-type translation $\tau_{x}^\alpha$ and the translation on the Bessel–Kingman hypergroups $\mathcal{T}_{|x|}^\alpha$  [\cite{spsa2}, Proposition $4.1$] , we get
    \begin{align}
      & \int_{B(0,r)} |\tau_x^\alpha g(y) -g^\alpha_{B(0,r)}(x) |^p d\mu_\alpha(y) \nonumber
  \\
    &\leq 2^{p-1} \left (  |g^\alpha_{B(0,r)}(x)|^p r^{d_\alpha} + 2^{3p+1}\int_{0}^r \mathcal{T}_{|x|}^\alpha |.|^{-\frac{p d_\alpha}{q}} (y)  d\mu_\alpha(y)  \right) \nonumber
     \\
    &= 2^{p-1} \left (  |g^\alpha_{B(0,r)}(x)|^p r^{d_\alpha} + 2^{3p+1} \int_{0}^\infty y^{-\frac{p d_\alpha}{q}}   \mathcal{T}_{|x|}^\alpha \chi_{J(0,r)}(y) d\mu_\alpha(y)  \right) \nonumber
    \\
     &= 2^{p-1} \left (  |g^\alpha_{B(0,r)}(x)|^p r^{d_\alpha} + 2^{3p+1} \int_{J(|x|, r)} y^{-\frac{p d_\alpha}{q}}  \mathcal{T}_{|x|}^\alpha \chi_{J(0,r)}(y)  d\mu_\alpha(y)  \right), \label{eq:6.2ab}
\end{align}
where $ J(|x|, r) = \{y \in (0, \infty) : y \in \,]max\{0, |x|-r\}, |x| + r[\,\} $  and we have used the support of $\mathcal{T}_{|x|}^\alpha\chi_{J(0,r)}$ (see \cite{Rachdi}). We consider two cases:
\begin{itemize}
    \item [\bf Case 1.] If $|x| > 2r$, then 
    \begin{align}
      J(|x|, r) & = \{y \in (0, \infty) : y \in \,]max\{0, |x|-r\}, |x| + r[\,\}\nonumber
       \\
       &= \{ y \in (0,\infty): y \in ]  |x| - r, |x| + r[\}. \nonumber 
    \end{align}
    This implies for $y \in J(|x|,r)$, $y >  |x|-r > r$. Then from \eqref{eq:6.2ab}, we have 
    \begin{align}
            &\int_{B(0,r)} |\tau_x^\alpha g(y) -g^\alpha_{B(0,r)}(x) |^p  d\mu_\alpha(y) \nonumber
            \\
            &\leq 2^{p-1} \left (  |g^\alpha_{B(0,r)}(x)|^pr^{d_\alpha} + 2^{3p+1} \int_{J(|x|, r)} y^{-\frac{p d_\alpha}{q}}  \mathcal{T}_{|x|}^\alpha \chi_{J(0,r)}(y) d\mu_\alpha(y)  \right)\nonumber
           \\
         &\leq 2^{p-1} \left (  |g^\alpha_{B(0,r)}(x)|^p  r^{d_\alpha} + 2^{3p+1} \int_{J(|x|, r)} r^{-\frac{d_\alpha p}{q}} \mathcal{T}_{|x|}^\alpha \chi_{J(0,r)}(y) d\mu_\alpha(y)  \right) \nonumber
           \\
         &\leq 2^{p-1} \left (  |g^\alpha_{B(0,r)}(x)|^p  r^{d_\alpha}+ 2^{3p+1} r^{-\frac{d_\alpha p}{q}} \int_0^\infty \mathcal{T}_{|x|}^\alpha \chi_{J(0,r)}(y) d\mu_\alpha(y)  \right) \nonumber
         \\
         &\leq 2^{p-1} \left (  r^{d_\alpha \left( 1- \frac{p}{q}\right)} + 2^{3p+1} r^{d_\alpha \left( 1- \frac{p}{q}\right)} \right) =   2^{p-1} (  1 + 2^{3p+1} ) r^{d_\alpha \left( 1- \frac{p}{q}\right)}, \label{eq:6.3ab}
    \end{align}
  where in the last inequality, we have used [\cite{Rachdi}, Proposition 2.3] and \eqref{eq:6.7a}.  Therefore substituting \eqref{eq:6.3ab} in \eqref{eq:6.12}, we obtain
    \begin{align}
      r^{d_\alpha\left( \frac{1}{q} - \frac{1}{p} \right)} \left(\int_{B(0,r)} |\tau_x^\alpha g(y) -g^\alpha_{B(0,r)}(x) |^p d\mu_\alpha(y)\right)^{\frac{1}{p}} \leq C < \infty. \nonumber
    \end{align}
    \item[\bf Case 2.] If $|x| \leq 2r$, then 
     \begin{align}
       J(|x|, r) & = \{y \in (0, \infty) : y \in \,]max\{0, |x|-r\}, |x| + r[\,\}\nonumber
      \\
       & \subseteq\{ y \in (0, \infty) : y \in ]  0, |x| + r[\} \subseteq J(0, 3r). \nonumber
    \end{align}
    Thus from \eqref{eq:6.2ab}, using the fact $\mathcal{T}_{|x|}^\alpha \chi_{J(0,r)}(y) \leq 1$ (see \cite{spsa2}), we get
    \begin{align}
            &\int_{B(0,r)} |\tau_x^\alpha g(y) -g^\alpha_{B(0,r)}(x) |^p  d\mu_\alpha(y) \nonumber
            \\
            &\leq 2^{p-1} \left (  |g^\alpha_{B(0,r)}(x)|^p r^{d_\alpha} + 2^{3p+1} \int_{J(|x|, r)} y^{-\frac{p d_\alpha}{q}}   d\mu_\alpha(y)  \right)\nonumber
           \\
         &\leq 2^{p-1} \left (  |g^\alpha_{B(0,r)}(x)|^p  r^{d_\alpha} + 2^{3p+1} \int_{J(0, 3r)} y^{-\frac{d_\alpha p}{q}}  d\mu_\alpha(y)  \right) \nonumber
           \\
         &= 2^{p-1} \left (  |g^\alpha_{B(0,r)}(x)|^p  r^{d_\alpha}+ 2^{3p+1} \frac{(3r)^{d_\alpha \left( 1 - \frac{p}{q} \right)}}{d_\alpha \left( 1 -\frac{p}{q} \right)} \right)  \nonumber
         \\
         &\leq 2^{p-1} \left ( C  + 2^{3p+1} \frac{3^{d_\alpha \left( 1 - \frac{p}{q} \right)}}{d_\alpha \left( 1 -\frac{p}{q} \right)} \right) r^{d_\alpha \left( 1- \frac{p}{q}\right)}=  C r^{d_\alpha \left( 1- \frac{p}{q}\right)}, \hspace{0.5cm} p < q, \label{eq:6.4ab}
    \end{align}
  using \eqref{eq:6.8a} in the last inequality. Finally, substituting \eqref{eq:6.4ab} in \eqref{eq:6.12}, we obtain
    \begin{align}
      r^{d_\alpha\left( \frac{1}{q} - \frac{1}{p} \right)} \left(\int_{B(0,r)} |\tau_x^\alpha g(y) -g^\alpha_{B(0,r)}(x) |^p d\mu_\alpha(y)\right)^{\frac{1}{p}} \leq C < \infty. \nonumber
    \end{align}
\end{itemize}
Thus combining both cases,  we conclude $||g||_{\mathcal{L}^{p,q}(\R, d\mu_\alpha)}$ is finite. Hence $g \in \mathcal{L}^{p, q}(\R, d\mu_\alpha)$. Also observe that
\begin{align}
     \mathcal{H}_{\psi}^\alpha g(x) = \int_\R \frac{\psi(t)}{|t|^{2\alpha +2}} \left|\frac{x}{t} \right|^{-\frac{d_\alpha }{q}}  dt = |x|^{-\frac{d_\alpha }{q}} \int_\R \frac{\psi(t)}{|t|^{d_\alpha \left( 1 -\frac{d_\alpha}{q} \right)}}   dt = K_\psi g(x),\nonumber
\end{align}
which implies
\begin{align}
    ||\mathcal{H}_\psi^\alpha g||_{\mathcal{L}^{p,q}(\R, d\mu_\alpha)}= K_\psi ||g||_{\mathcal{L}^{p,q}(\R, d\mu_\alpha)}. \nonumber
\end{align}
Hence $K_\psi < \infty$.
This completes the proof.  \qed

 \section*{
  Declarations} 
\textbf{Competing Interests} The author declares no competing interests\\

\textbf{Ethical Approval} 
    Not applicable  \\
    

\textbf {Data availability statement} Not applicable\vspace{3mm}



\bigskip  

\small 
\noindent
{\bf Publisher's Note}
Springer Nature remains neutral with regard to jurisdictional claims in published maps and institutional affiliations. 
\end{document}